\documentclass[12pt, reqno]{amsart}
\usepackage{graphicx} 
\usepackage{dsfont}
\usepackage{amsmath}
\usepackage{amscd}
\usepackage{amsthm}
\usepackage{amssymb} \usepackage{latexsym}
\usepackage{eufrak}
\usepackage{euscript}
\usepackage{epsfig}
\usepackage{graphics}
\usepackage{array}
\usepackage{enumerate}
\usepackage{color}
\usepackage{wasysym}
\usepackage{pdfsync}
\usepackage{stmaryrd}

\newcommand{\C}{\mathbb{C}}
\newcommand{\R}{\mathbb{R}}
\newcommand{\Z}{\mathbb{Z}}
\newcommand{\N}{\mathbb{N}}

\newcommand{\cT}{{\mathcal T}}
\newcommand{\bX}{\mathbb{X}}
\newcommand{\cuad}{{\sqcap\kern-.68em\sqcup}}

\newcommand{\cM}{{\mathcal M}}

\newtheorem{theorem}{Theorem}[section]
\newtheorem{proposition}{Proposition}[section]

\newtheorem{lemma}{Lemma}[section]
\newtheorem{corollary}{Corollary}[section]
\newtheorem{remark}{Remark}[section]
\newcommand{\bremark}{\begin{remark} \em}
\newcommand{\eremark}{\end{remark} }

\newtheorem{problem}{Problem}[section]

\begin{document}

\title[ Kazdan-Warner equations on the lattice]{On finite-energy solutions of Kazdan-Warner type equations on the lattice graph}


\author{Huyuan Chen}
 \address{Huyuan Chen: Center for Mathematics and Interdisciplinary Sciences, Fudan University,  
Shanghai 200433, China.\qquad 
Shanghai Institute for Mathematics and Interdisciplinary Sciences, 
 Shanghai 200433, PR China.
}
\email{chenhuyuan@yeah.net, chenhuyuan@simis.cn}

\author{Bobo Hua}
\address{Bobo Hua: School of Mathematical Sciences, LMNS,
Fudan University, Shanghai 200433, PR China.}
\email{bobohua@fudan.edu.cn}

\begin{abstract}
We  investigate  finite-energy  solutions to Kazdan-Warner type equations  in 2-dimensional integer lattice graph
$$    - \Delta u= \varepsilon e^{\kappa  u} +\beta\delta_0\quad{\rm in}\ \,  \Z^2,$$
where  $\varepsilon=\pm1$, $\kappa>0$ and  $\beta\in\R$.

When $\varepsilon=1$, we prove the existence of a continuous family of finite-energy solutions for some parameter $\kappa$. This resolves the open problem on the existence of finite-energy solutions to the Liouville equation, i.e. $\kappa=1,\beta=0.$
 
 When $\varepsilon=-1$ and $\beta>\frac{4\pi}{\kappa}$, we prove that the set of finite-energy solutions exhibits a layer structure. Moreover, we derive the extremal solution in this case.
\end{abstract}


\vspace{2mm}



\maketitle

   \smallskip

   \noindent {\small {\bf AMS Subject Classifications}:  35J91,  05C22    
  
  \smallskip
	\medskip 
\section{Introduction}

The Kazdan-Warner equation is a nonlinear partial differential equation that arises in differential geometry for solving the prescribed Gaussian curvature problem on surfaces in conformal geometry.
Let $(\cM,g)$ be a Riemannian surface and let $\tilde{g} = g e^{2u}$ be a metric with conformal change $u\in C^2(\cM)$. 
For the prescribed Gaussian curvature problem, we have the following equation,  introduced by Kazdan and Warner \cite{KW},
  \begin{equation}\label{eq 1.0-ga}
-\Delta_g u = \tilde{K} e^{2u}-K, 
 \end{equation}
where $\Delta_g$ is the Laplace-Beltrami operator, $K$ and $\tilde{K}$ are the Gaussian curvature of $(\cM,g)$ and $(\cM,\tilde{g})$ respectively.  
These type equations have been the subjects of extensive investigation in the literature \cite{Moser1970,Moser1973,Aubin1979,ding1997differential,Chang2004}, 
motivated by its connections to geometry, statistical mechanics,  and physics; see also the mean field equation \cite{Onsager1949,Eyink2006} and the Chern-Simons equation \cite{jaffee1980vortices,hong1990multivortex} etc.

Consider the special case for conformal changes on the Euclidean plane $\R^2.$ For finding a conformal metric $(\R^2, e^{\kappa u} ds_{\R^2}^2)$ with constant Gaussian curvature $\frac{\kappa}{2}$ with $\kappa>0,$ we have the Liouville equation 
\begin{equation}\label{eq 1.1-liouville}
-\Delta u=e^{\kappa u}\quad {\rm in}\ \ \R^2.
\end{equation} Although one may reduce the general parameter $\kappa$ to the model case $\kappa = 2$ by the scaling of $\R^2,$  we will keep it for our purposes where the scaling is absent. For any solution $u$ of \eqref{eq 1.1-liouville},  by a classical result of Liouville
\cite{Liouville}, any solution $u$ of \eqref{eq 1.1-liouville}
has the form
 $$u(z)=\frac{1}{\kappa}\ln \frac{8|f'|^2}{(1+\kappa |f|^2)^2},$$
where $f$ is a locally univalent meromorphic function on $\C.$ By a seminal result of Chen–Li \cite{CLi1}, the solution $u$ of \eqref{eq 1.1-liouville} satisfies $\int_{\R^2} e^{\kappa u}dx<+\infty$, i.e., it has finite total curvature or finite energy, if and only if there exist $\lambda>0, x_0\in \R^2$ such that
 \begin{equation}\label{eq:expsol}u(x)=\frac1\kappa \ln \frac{32 \lambda^2}{4+ \lambda^2\kappa |x-x_0|^2}.\end{equation} Moreover, in this case the total curvature 
 \begin{equation}\label{eq:totcurv}
 \int_{\R^2} \frac{\kappa }{2}e^{\kappa u}dx=4\pi.\end{equation} See \cite{CW,Chanillo1995,JostWang2002,Brito2004,HangWang2006,Eremenko2007,LiTang2020} for other proofs, and a recent survey \cite{CaiLai2024}.

In recent years, the study of partial differential equations on graphs has attracted increasing attention and has played an important role in various fields. For a given graph, there is a natural discrete Laplace operator, which allows one to formulate and analyze PDE problems in the discrete setting. In the case of nonlinear equations, Grigor’yan, Lin, and Yang \cite{grigor2016kazdan} investigated the Kazdan–Warner equation on finite graphs and established existence results, which were subsequently extended in \cite{ge2017kazdan,fang2018class,ge2020p,liu2020multiple,camilli2022note,pinamonti2022existence,sun2022brouwer,li2023existence,zhang2024fractional}. Existence results for Kazdan–Warner equations on infinite graphs were proven in \cite{ge2018kazdan,keller2018kazdan} under certain assumptions. Furthermore, Huang, Lin, and Yau \cite{huang2020existence} proved the existence of solutions to the mean field equation and the Chern–Simons equation on finite graphs. A number of related works have also been devoted to exponential nonlinearities on graphs; see, for example, \cite{huang2021mean,lu2021existence,hou2022existence,hua2023existence,lin2022calculus,chao2023multiple,gao2023existence,hou2024topological,li2024topological}.


In this paper, we investigate  the existence of solutions of nonlinear equations on the 2-dimensional lattice graph $\Z^2,$ which consists of the set of vertices $\mathbb{Z}^2=\big\{x=(x_1,x_2):x_1,x_2\in \Z\big\}$ and the set of edges $E=\{\{x,y\}: x,y\in \Z^2, |x-y|=1\}.$ Lattice graphs are commonly used in statistical physics and numerical analysis, and they provide a natural framework for studying problems that arise in the theory of partial differential equations in the discrete setting.  For simplicity, we write $x\sim y$ if $\{x,y\}\in E.$  The discrete Laplacian is defined as, for any $u:\Z^2\to \R,$ $$\Delta u(x):=\sum_{y\in \Z^2: y\sim x}\big(u(y)-u(x)\big),\ x\in\mathbb{Z}^2.$$  


In this paper, we consider Kazdan–Warner type equations on $\Z^2,$
  \begin{equation}\label{eq 1.0}
 -\Delta u=  \varepsilon  e^{\kappa u}+\beta\delta_0\quad {\rm in}\ \, \Z^2,
 \end{equation}
where    $\varepsilon=\pm1$,  $\beta\in\R,\,  \kappa>0$ and $\delta_0$ is the indicator function (the delta mass) at the zero. It is called the \emph{(discrete) Liouville equation} for $\epsilon=1$ and $\beta=0.$
 For general $\beta\neq 0,$ the solution $u$ stands as an analogue of Green's function to the nonlinear equation.

When $\varepsilon=1$, the above equation in the manifold $\cM$ takes the form
\begin{equation}\label{eq 1.0-g}
-\Delta u = e^{u} + \beta \delta_0 \quad \text{in } \cM,
\end{equation}
where the Dirac mass $\delta_0$ is interpreted as a source term or singularity in physical models. The parameter $\beta$ specifies the strength of this singularity and is closely connected to the mean field equation, which arises in both geometry and statistical physics, e.g. the mean field limit of the Euler flow in fluid dynamics \cite{CLMP92}.  
Important results have been established concerning the existence of solutions and their qualitative properties, with significant implications for both theoretical physics and the analysis of nonlinear partial differential equations. For example, in \cite{CL03,CLW04}, it was shown that Eq.(\ref{eq 1.0-g}) always admits solutions whenever $\beta = 8m\pi$ for any integer $m$. Moreover, \cite{LW10} proves that for $\beta = 8m\pi$, Eq.(\ref{eq 1.0-g}) admits solutions if and only if the Green function on the two-dimensional flat torus possesses critical points other than the three half-period points. For related studies on the two-dimensional flat tori, we refer the reader to \cite{M79,LY13}.  
These results are essential for understanding the structure of solutions to mean field equations and their interplay with the geometric properties of the underlying space. This motivates our studies in the discrete setting in $\Z^2.$


For the Liouville equation on $\Z^2$, Ge, Hua, and Jiang \cite{ge2018note} proved that there exists a universal constant $c>0$ such that every solution $u$ satisfies \[ \int_{\Z^2} e^{u} \geq c, \] where the integral is with respect to the counting measure, same as the summation.  However, we don't know the existence of finite-energy solutions of the Liouville equation on $\Z^2.$ Note that on $\R^2$ the existence of finite-energy solutions follows directly from their explicit representation, which are radially symmetric with respect to some point, while no such explicit construction is expected in the discrete setting. Moreover, the existence results obtained in \cite{ge2018kazdan,keller2018kazdan} rely on strong assumptions concerning either the infinite graph or the known function, which are not applicable to the Liouville equation on the lattice. Hence, establishing the existence of finite-energy solutions to the Liouville equation on $\Z^2$ constitutes one of fundamental open problems in this area. 

\begin{problem}[{\cite[Problem~2]{ge2018note}}]\label{prob2}
    Is there any solution to the Liouville equation  on $\Z^2$ with finite energy, i.e., $\int_{\Z^2} e^{u}<\infty.$
\end{problem}



When $\varepsilon=1$, the nonlinear term $ e^{\kappa u}$ is a source, the equation (\ref{eq 1.0}) becomes 
 \begin{equation}\label{eq 1.1}
 -\Delta u=    e^{\kappa u}+\beta\delta_0\quad\ 
    {\rm in}\ \  \Z^2
 .\end{equation}
We prove the existence of finite-energy solutions to the above equation for small parameter $\kappa.$ In the literature, a solution $u$ is called a \emph{non-topological} solution if $\lim_{x\to \infty} u(x)=-\infty. $ Note that finite-energy solutions are always non-topological solutions in the discrete case.

\begin{theorem}\label{teo 1}
\begin{enumerate}
\item[$(i)$] There are universal constants $\kappa^*>0$ and $a_0>2$ such that for any $\kappa\in (0,\kappa^*],$ and any $\beta\in [0,\frac{2\pi a_0}{\kappa}),$ there exists a solution $u$ to the equation \eqref{eq 1.1} satisfying 
\begin{align}\label{ep-1}
u(x)=-\frac{a_0}{\kappa}\ln|x|+d_{\kappa,\beta}+O\big(|x|^{\frac{2-a_0}{a_0+1}}\big(\ln |x|\big)^{\frac{1}{a_0+1}}\big) 
\quad {\rm as}\ \,  x\to\infty,
\end{align}
where  $d_{\kappa,\beta}\in\R$ depends on $\kappa,\beta$. 
Moreover,  \begin{align}\label{ep-1-e22}
\int_{\Z^2} e^{\kappa u} =\frac{2\pi a_0}{\kappa}-\beta.
\end{align}
\item[$(ii)$] For any $\epsilon\in(0,1)$, there exists $\overline{\kappa}=\overline{\kappa}(\epsilon)>\kappa^*$ such that for any  $\kappa\in (0,\overline{\kappa}],$ any $\alpha\in\big[\frac{2\pi(2+\epsilon)}{\kappa},\frac{2\pi(2+\epsilon^{-1})}{\kappa}\big]$ and any $\beta\in[0,\alpha),$ 
there exists a solution $u_{\kappa,\alpha,\beta}$ to the equation \eqref{eq 1.1} satisfying 
\begin{align}\label{ep-1}
u_{\kappa,\alpha,\beta}(x)=-\frac{\alpha}{2\pi}\ln|x|+d_{\kappa,\alpha,\beta}+O\big(|x|^{\frac{4\pi-\alpha\kappa}{\alpha\kappa+2\pi}}\big(\ln |x|\big)^{\frac{2\pi }{\alpha\kappa+2\pi}}\big) 
\quad {\rm as}\ \,  x\to\infty,
\end{align}
where $d_{\kappa,\alpha,\beta}\in\R$ depends on $\kappa,\alpha,\beta$. Moreover, 
 \begin{align}\label{ep-1-e22}
\int_{\Z^2} e^{\kappa u_{\kappa,\alpha,\beta}} =\alpha-\beta.
\end{align} 
\end{enumerate}
 \end{theorem}

Since no radial solutions are expected on $\Z^2,$ we prove the results by using Schauder's fixed point theorem. We first prepare proper Banach spaces, $\ell^\infty_\sigma(\Z^2)$ for $\sigma>0,$ weighted with respect to the $\sigma$-th power of the distance function; see \eqref{eq:wsob}. The key ingredient is the compact embedding of $\ell^\infty_{\sigma_1}(\Z^2)\hookrightarrow \ell^\infty_{\sigma_2}(\Z^2)$  for $\sigma_1>\sigma_2;$ see Lemma~\ref{lm 2.1}. Using the precise asymptotics of the discrete Green's function $\Phi_0$ on $\Z^2$ by Kenyon \cite{Ke}, we prove a crucial estimate for $\Phi_0*f$ for $f\in \ell^\infty_\sigma(\Z^2)$ with $\sigma>2;$ see Proposition~\ref{pr 2.1}. Then we apply the arguments of fixed point theorem to conclude the existence of finite-energy solutions by the smallness of $\kappa$. 

Consider the special case of the Liouville equation, i.e. $\beta=0$, 
\begin{equation}\label{eq 1.1-0}
 -\Delta u=   e^{\kappa u} \quad\
    {\rm in}\ \  \Z^2.\end{equation} By Theorem~\ref{teo 1}, we prove the existence of finite-energy solutions for the Liouville equation, resolving Problem~\ref{prob2}.

\begin{theorem}\label{thm:main0}
    For any $a>4\pi,$ there exists a solution $u_{a}$ to the Liouville equation \eqref{eq 1.1-0} with $\kappa=1$ satisfying 
\begin{align}\label{ep-1}
u_{a}(x)=-\frac{a}{2\pi}\ln|x|+c_a+O(|x|^{\frac{4\pi-a}{a+2\pi}}\big(\ln |x|\big)^{\frac{2\pi }{a+2\pi}}) 
\quad {\rm as}\ \,  x\to\infty,
\end{align}
where $c_a\in\R$ depends on $a.$ Moreover, $\int_{\Z^2} e^{u_a} =a<\infty.$
\end{theorem}


\begin{remark}
For the Liouville equation with $\kappa=1$ on $\R^2,$ by \eqref{eq:totcurv} it is known that any finite-energy solution satisfies $\int_{\R^2}e^{u}=8\pi$. In contrast, our result shows that on $\Z^2$ there exists a continuous family of finite-energy solutions whose total energy ranges over $(4\pi,\infty).$ Moreover, while any finite-energy solution in $\R^2$ has the asymptotic behavior$$u(x)\sim -2\ln|x|,\quad x\to \infty,$$ our solutions on $\Z^2$ exhibit the varying asymptotics, i.e. for all $c\in(2,\infty)$ there exists some solution
$$u(x)\sim -c\ln|x|,\quad x\to \infty.$$ These phenomena could be caused by the geometry of the lattice, e.g., Green's function of the Laplacian in $\Z^2$ is non-positive, which is a new feature. This provides rich results, exploring how the discrete nature of the graph interacts with the nonlinearities in the equation.
\end{remark}



 When $\varepsilon=-1,$  it is known that   the nonlinear term $ e^{\kappa u}$ is an absorption, and equation (\ref{eq 1.0}) turns to be 
\begin{equation}\label{eq 1.1-ab}
 -\Delta u+  e^{\kappa u} =\beta\delta_0\quad
    {\rm in}\ \  \Z^2. 
\end{equation} In this case, the solution to the equation satisfies the maximum principle. We adopt the framework of the fixed point theorem to prove the existence of  finite-energy solutions.

\begin{theorem}\label{teo 1-ab}
Let   $\kappa>0$,  $\alpha_0=\frac{4\pi}{\kappa}$ and  $\beta>\alpha_0$. 
 
$(i)$  For any $\alpha\in  \big(\alpha_0, \beta\big)$ problem (\ref{eq 1.1-ab})
has a solution ${\bf u}_\alpha$ satisfying 
\begin{align}\label{ep-1-ab}
{\bf u}_{\alpha}(x)=-\frac{\alpha}{2\pi}\ln|x|+{\bf d}_{\kappa,\alpha,\beta}+O(|x|^{\frac{4\pi-\alpha\kappa}{\alpha\kappa+2\pi}}\big(\ln |x|\big)^{\frac{2\pi }{\alpha\kappa+2\pi}}) 
\quad {\rm as}\ \, x\to\infty,
\end{align}
where ${\bf d}_{\kappa,\alpha,\beta}\in\R$ depends on $\alpha,\beta$ such that 
$${\bf d}_{\kappa,\alpha,\beta}\leq \frac1\kappa \ln(\beta-\alpha)-\frac{\gamma_0}{2}\alpha, $$ for $\gamma_0=\frac{1}{\pi}(\gamma_E+\frac12\ln 2)$ with the Euler constant $\gamma_E.$

$(ii)$   Problem (\ref{eq 1.1-ab})
has a solution ${\bf u}_{\alpha_0}$ satisfying 
\begin{align}\label{ep-1-0-ab}
{\bf u}_{\alpha_0} (x)=-\frac{2}{\kappa}\ln|x|-\frac{2}{\kappa}\ln\ln|x|+O(1)
\quad {\rm as}\ \, x\to\infty.
\end{align}

$(iii)$  the mapping $\alpha\in[\alpha_0,\beta) \mapsto {\bf u}_\alpha$ is strictly decreasing, continuous locally in $\Z^2$, i.e.
$${\bf u}_{\alpha_1}\geq  {\bf u}_{\alpha_2} \quad{\rm in}\ \, \Z^2\qquad{\rm for}\ \  \alpha_0\leq \alpha_1\leq \alpha_2< \beta, $$
$${\bf u}_{\tilde \alpha}=\lim_{\alpha\to\tilde \alpha} {\bf u}_\alpha\quad{\rm in}\ \, \Z^2\qquad{\rm for}\ \  \tilde\alpha\in (\alpha_0,\beta), $$
$${\bf u}_{\alpha_0}=\lim_{\alpha\to\alpha_0^+}{\bf u}_\alpha\quad{\rm in}\ \, \Z^2$$
 and  for $\alpha\in[\alpha_0,\beta)$
\begin{align}\label{ep-1-e-ab}
\int_{\Z^2} e^{\kappa {\bf u}_\alpha} dx=\beta-\alpha.  
\end{align}
 
   \end{theorem}
 
 We remark that the constant ${\bf d}_{\alpha,\beta}\to-\infty$ as $\alpha \to \beta^-$. On the other hand, ${\bf d}_{\alpha,\beta}$ remains bounded, and the decay rate of $|x|^{\frac{4\pi-\alpha\kappa}{\alpha\kappa+2\pi}}\big(\ln |x|\big)^{\frac{2\pi }{\alpha\kappa+2\pi}}$ vanishes for $\alpha\to \alpha_0^+$.  Inspired by the continuous model,  when $\alpha=\alpha_0$  the solution $u_{\alpha_0}$ is called \emph{extremal solutions}.  Note that in this critical case, by  \eqref{ep-1-0-ab}, the solution exhibits a double logarithmic asymptotic behavior, which also appears in the continuous setting. Here, we obtain {\it a layer structure} of the set of solutions $\{{\bf u}_\alpha\}_{\alpha\in[\alpha_0,\beta)}$.  
  \smallskip
  
  {
  
 Finally,   we  add a remark to address the dependence of the problems on the parameters $\varepsilon, \kappa$ and $\beta.$ 

 \begin{remark}
 (1) For any $\varepsilon \in (-1,0) \cup (0,1)$, define the rescaling $v = \frac{1}{|\varepsilon|} u$. Then the original Kazdan–Warner type problem (\ref{eq 1.0})  
transforms into  
$$
-\Delta v = \operatorname{sgn}(\varepsilon) \, e^{\tilde{\kappa} v} + \tilde{\beta} \, \delta_0 \quad \text{in } \mathbb{Z}^2,
$$  
where $\tilde{\kappa} = |\varepsilon| \kappa$ and $\tilde{\beta} = \frac{\beta}{|\varepsilon|}$. This rescaled problem falls within the scope of our main results.  

(2) It is significantly more challenging—both analytically and structurally to consider the non-topological solutions of (\ref{eq 1.0})  —if either the constant $\kappa$ is replaced by a nonnegative vertex-dependent function $\kappa: V \to [0,+\infty)$, or the point source $\beta \delta_0$ is replaced by a general nonnegative function $\beta: V \to [0,+\infty)$. We leave these as open questions for future research.
\end{remark}
 }

The structure of this paper is as follows. In Section 2, we investigate the fundamental properties, including the comparison principle, embedding results, and related estimates. In Section 3, we establish the existence of solutions in the source case using fixed-point theory.    Finally, Section 4 focuses on the absorption case, where we employ the Leray-Schauder theorem to prove the existence of  finite-energy solutions and utilize the method of super and sub-solutions to obtain the extremal solution.

    \setcounter{equation}{0}
\section{ Preliminaries  }

\subsection{Comparison principle and Embedding results}
 
Let $G=(V,E)$ be a (possibly infinite) simple, locally finite, and undirected graph. Two vertices $x,y$ are called neighbours, denoted by $x\sim y,$ if there is an edge connecting $x$ and $y.$ A subset $\Omega\subset V$ is called connected if for any $x,y\in \Omega,$ there is a path $\{x_i\}_{i=0}^n\subset \Omega$ from $x$ to $y,$ i.e. $x_i\sim x_{i+1}$ for any $0\leq i\leq n-1$ and $x_0=x,x_n=y.$ The combinatorial distance between two vertices $x$ and $y,$ $d(x,y),$ is the length of the shortest path connecting $x$ and $y$. We usually write
$x\to \infty$ if $d(x,x_0)\to \infty$ for a fixed vertex $x_0.$  
For $\Omega\subset V,$ we denote by $\delta\Omega:=\{y\in V\setminus \Omega: \exists x\in \Omega \ \mathrm{s.t.}\ y\sim x\}$ the boundary of $\Omega.$ We write $\overline{\Omega}:=\Omega\cup\delta\Omega.$  Moreover, we use the following notations: $|x-\bar x|_{Q}=|x_1-\bar x_1|+|x_2-\bar x_2|$,   $|x-\bar x|:=\sqrt{(x_1-\bar x_1)^2+(x_2-\bar x_2)^2}$ $$ Q_r(x_0)=\big\{x\in\Z^2:\, |x-x_0|_{_Q}\leq r  \big\},\qquad B_r(x_0)=\big\{x\in\Z^2:\, |x-x_0| \leq r \big\}.$$
For any $\Omega\subset V$ and any $u:\overline{\Omega}\to \R,$ the Laplacian of $u$ is defined as  $$\Delta u(x):=\sum_{y\sim x}\big(u(y)-u(x)\big)\quad\ \text{ for all }x\in\Omega.$$
 
 The following maximum principle is well known in the continuous setting, for which we give the proof in the discrete setting.
 \begin{theorem}\label{thm:max}
     For a graph $G=(V,E)$ and a (possibly infinite) connected subset $\Omega\subset V$   verifying
   either $\delta \Omega\not=\emptyset$ or $\Omega$ is unbounded,  
     if  $u: \overline{\Omega}\to \R$ satisfies
\begin{equation}\label{eq 2.1 cm}
\left\{
\begin{array}{lll}
-\Delta u+c u     \geq  0   \quad
   &{\rm in}\ \  \Omega , \\[1mm]
 \phantom{ ----  }
u\geq 0 \quad &{\rm   in}\ \ \,    \delta\Omega, \\[1mm]
\liminf\limits_{d(x,0)\to \infty, x\in \Omega}u(x)\geq 0,
 \end{array}
 \right.
\end{equation} where $c:\Omega\to[0,\infty),$
then $u\geq 0$ in $\Omega$.  Furthermore, $u\equiv 0$ in $\Omega$ or $u>0$ in $\Omega$. 
 \end{theorem}
\begin{proof}
 Without loss of generality, we prove the result for an infinite subset $\Omega.$ Since $\Omega$ is connected, so is $\overline{\Omega}.$ Suppose that the first assertion is not true, i.e. there exists $x_0\in\Omega$ such that $u(x_0)<0.$ Since $\liminf\limits_{x\to \infty, x\in \Omega}u(x)\geq 0
$ and $u|_{\delta\Omega}\geq 0,$ then $\inf_{x\in \overline{\Omega}}u<0$ and $$A:=\{x\in \overline\Omega:u(x)=\inf_{x\in\overline\Omega}u\}\neq\emptyset,\quad A\subset \Omega.$$ For any $x\in A,$ we claim that $y\in A$ for any $y\sim x, y\in \overline\Omega.$ By the equation,
    $$\Delta u(x)\leq c(x)u(x)\leq 0.$$ Since $u(x)=\min_{x\in \overline\Omega}u,$ this proves the claim. By the connectedness of $\overline\Omega,$ $A=\overline\Omega,$ which contradicts $A\subset \Omega.$ This proves the first assertion.

    For the second assertion, if there exists $x_0\in\Omega$ such that $u(x_0)=0,$ then by the same argument above, one can show that $u\equiv 0$ on $\overline\Omega$.     This proves the result. 
\end{proof}

\begin{remark}
When $V$ is finite, connected and $\Omega=V$, the comparison principle holds when the boundary conditions  are replaced by $u(x_0)\geq0$ for some $x_0\in\Omega$. 
\end{remark}
 
 {
  
  \begin{corollary}\label{cr com}
     For a graph $G=(V,E)$ and a (possibly infinite) connected subset $\Omega\subset V$   verifying
   either $\delta \Omega\not=\emptyset$ or $\Omega$ is unbounded and   
     if  $u, v: \overline{\Omega}\to \R$ satisfy
\begin{equation}\label{eq 2.1 cm+}
\left\{
\begin{array}{lll}
-\Delta u+f(x,u)     \geq  -\Delta v+f(x,v),   \quad
   &{\rm in}\ \  \Omega , \\[1mm]
 \phantom{ -----  }
u-v\geq 0 \quad &{\rm   in}\ \ \,    \delta\Omega, \\[1mm]
\liminf\limits_{d(x,0)\to \infty, x\in \Omega}\big(u(x)-v(x)\big)\geq 0,
 \end{array}
 \right.
\end{equation}
  where $f(x,t):\Omega\times \R \to[0,\infty)$ and  $\partial_t f(x,t)\geq 0$ for any $x\in \Omega$. 
Then $u\geq v$ in $\Omega$.    
 \end{corollary}
\begin{proof} Let $w=u-v$, then 
$$-\Delta w +c(x) w\geq 0\quad {\rm in}\ \, \Omega, $$
where 
$$c(x)=\left\{
\begin{array}{lll}
 \frac{f(x,u(x))-f(x,v(x)) }{u(x)-v(x)}    \quad
   &{\rm if}\ \ u(x)\not=v(x), \\[2mm]
 \phantom{   }
 \partial_t f(x,u(x))  \quad &{\rm if}\ \ u(x)=v(x). 
 \end{array}
 \right.
$$
which is a nonnegative function in $\Omega$. 
 We now apply Theorem \ref{thm:max} to obtain that $w\geq0$ in $\Omega$. 
\end{proof}
}
 
  Given $\Omega\subset \Z^2$, without any confusion,  we use the notation 
 $$\int_{ \Omega} f(x)dx =\sum_{x\in \Omega} f(x)\quad{\rm for}\ \  f: \Omega\to \R. $$
 For the integer lattice $\Z^2,$ we denote by $\ell^\infty(\Z^2)$ the space of bounded functions on $\Z^2$ with the norm $\|\cdot\|_{\ell^\infty},$ and for any $\sigma\in \R,$ we define the weighted $\ell^\infty_\sigma$ 
 \begin{equation}\label{eq:wsob}
     \ell^\infty_\sigma(\Z^2):=\big\{f:\Z^2\to\R: \|f\|_{\ell^\infty_\sigma}<\infty\big\}, 
 \end{equation}
 where 
$$\|f\|_{\ell^\infty_\sigma}:=\sup_{x\in \Z^2}\big(|f(x)|(1+|x|)^\sigma\big)$$
 is the weighted $\ell^\infty_\sigma$ norm. Note that  $\ell^\infty(\Z^2)=\ell^\infty_0(\Z^2).$

For $x=(x_1,x_2)\in \R^2,$ we write $|x|_p:=(|x_1|^p+|x_2|^p)^{\frac1p}$ for $p\in [1,\infty)$ and $|x|_\infty:=\max\{|x_1|,|x_2|\}.$  We always write $|x|=|x|_2.$
We denote by $B_r(x):=\{x\in \R^2: |x|\leq r\}$ the closed disc of radius $r$ centered at $x,$ and by $Q_r(x):=\{x\in \R^2: |x|_\infty\leq \frac{r}{2}\}$ the square of side length $r$ centered at $x.$ For $x=0,$ we write $B_r=B_r(0)$ for simplicity.

 \smallskip


 \begin{lemma}\label{lm 2.1}
 For $\sigma_1,\sigma_2\in \R$ with  $\sigma_1>\sigma_2,$ 
$$\ell^\infty_{\sigma_1}(\Z^2)\subset \ell^\infty_{\sigma_2}(\Z^2),\quad \mathrm{and} $$
$$\|u\|_{\ell^\infty_{\sigma_2}}\leq \|u\|_{\ell^\infty_{\sigma_1}},\quad \forall\ u\in \ell^\infty_{\sigma_1}(\Z^2).$$
Moreover, the embedding $\ell^\infty_{\sigma_1}(\Z^2)\hookrightarrow \ell^\infty_{\sigma_2}(\Z^2)$ is compact.

 \end{lemma}
\noindent{\bf Proof. }  The first assertion is obvious from the definition of space $\ell^\infty_{\sigma}(\Z^2)$. We prove the second assertion, i.e. the compactness of the embedding $\ell^\infty_{\sigma_1}(\Z^2)\hookrightarrow \ell^\infty_{\sigma_2}(\Z^2)$. Let $\{\varphi_n\}_{n\in\N}\subset \ell^\infty_{\sigma_1}(\Z^2)$ such that $\|\varphi_n\|_{\ell^\infty_{\sigma_1}}\leq 1.$
Note that for any $r\geq 1,$ $$\sup_n\max_{B_r\cap \Z^2}|\varphi_n|\leq C(r,\sigma_1)<\infty.$$ Hence, up to a subsequence, still denoted by $\varphi_n,$ there is a function $\varphi$ such that $\varphi_n\to \varphi$ pointwise on $\Z^2$ as $n\to\infty.$ One easily sees that
$$\|\varphi\|_{\ell^\infty_{\sigma_1}}\leq 1.$$ We need to prove that $\varphi_n\to \varphi$ in $\ell^\infty_{\sigma_2}(\Z^2)$ as $n\to\infty.$ This will prove the compactness.

For any $\epsilon>0$,  there exists $N_0\gg1$ such that for $|x|>N_0,$
\begin{align*}
 \Big( |\varphi_n(x)|+ |\varphi(x)|\Big) (1+|x|)^{\sigma_2}&= \Big( |\varphi_n(x)|+ |\varphi(x)|\Big)(1+|x|)^{\sigma_1} (1+|x|)^{\sigma_2-\sigma_1}
  \\& \leq \Big(\|\varphi_n\|_{\ell^\infty_{\sigma_1}}+\|\varphi\|_{\ell^\infty_{\sigma_1}}\Big)(1+N_0)^{\sigma_2-\sigma_1}   
   \\& \leq 2 \epsilon. 
\end{align*}
Moreover, for sufficiently large $n,$ \begin{align*}
  |\varphi_n(x)-\varphi(x)| &\leq  (1+N_0)^{-\sigma_2} \epsilon,\  \text{$\forall\  x\in B_{N_0}\cap \Z^2$}. 
\end{align*}
Therefore, we obtain that 
\begin{align*}
\sup_{x\in\Z^2} \Big( |\varphi_n(x)-\varphi(x)| (1+|x|)^{\sigma_2}\Big)&\leq (1+N_0)^{\sigma_2} \sup_{x\in  B_{N_0}\cap \Z^2} \Big( |\varphi_n(x)-\varphi(x)|\Big)
\\&\qquad + \sup_{x\in \Z^2\setminus B_{N_0}} \Big(\big( |\varphi_n(x)|+ |\varphi(x)|\big) (1+|x|)^{\sigma_2} \Big)
  \\& \leq  3\epsilon.
\end{align*}
This proves the convergence in $\ell^\infty_{\sigma_2}(\Z^2),$ and yields the result.   \hfill$\Box$\medskip

The following is an important property for the analysis on the graph $\Z^2,$ which follows from the parabolicity or the recurrence of the simple random walk on $\Z^2;$ see e.g. \cite{Gri18}.
\begin{theorem}\label{thm:para}
    Let $u:\Z^2\to \R$ satisfy $\Delta u\leq 0$ and $u\geq 0$ on $\Z^2.$ Then $u$ is constant.
\end{theorem}
 
\subsection{Some useful estimates}
The convolution of two functions $f,g: \Z^2\to\R$ is defined as 
$$\big(f\ast g\big)(x)= \int_{\Z^2} f(y-x) g(y)dy \quad {\rm for}\ \, x\in\Z^2,$$
whenever it makes sense.

We recall the asymptotic behavior of fundamental solution of 
 \begin{equation}\label{eq 2.1}
\left\{
\begin{array}{lll}
-\Delta u = \delta_0 \quad
   &{\rm in}\ \  \Z^2, \\[2mm]
 \phantom{ \   }
  u(0)=0,
 \end{array}
 \right.
\end{equation} which is Green's function on $\Z^2.$
It was proven by \cite[Theorem 7.3]{Ke} that the problem (\ref{eq 2.1}) has a unique solution 
  $\Phi_0$   satisfying 
 \begin{equation}\label{fund-1}
  \Phi_0(0)=0,\quad \Phi_0(x)= -\frac1{2\pi}\ln |x|-\frac{\gamma_0}{2}+O(|x|^{-1})\quad {\rm as}\ \ x\to\infty,
  \end{equation}
 where $\gamma_0=\frac{1}{\pi}(\gamma_E+\frac12\ln 2)$ with the Euler constant $\gamma_E.$   
 Moreover, we have $\Phi_0<0$ on $\Z^2\setminus\{0\}.$ Hence, there exists a universal constant $c_1\geq 1$ such that for any $x\in \Z^2,$
 \begin{equation}\label{eq:q1}-\frac1{2\pi} \ln(1+ |x|)-c_1\leq   \Phi_0(x) \leq  -\frac1{2\pi} \ln(1+ |x|)+c_1,\quad \mathrm{and}\end{equation}

 \begin{equation}\label{fund-2}
  -c_1 \ln(1+ |x|) \leq   \Phi_0(x) \leq  -\frac1{c_1} \ln(1+ |x|). 
  \end{equation}

\begin{proposition}\label{pr 2.1}
Let $f\in \ell^\infty_m(\Z^2)$ with  $m>2$ satisfy that 
\begin{align}\label{sst-1}
\int_{\Z^2} f(x)dx =0. 
\end{align}
  
Then there exists a universal constant $c_0>1$ such that 
 \begin{align}\label{sst-2}
 \big|\big(\Phi_0\ast f\big)(x)\big|\leq \frac{c_0^m }{(m-2)^4} \|f\|_{\ell^\infty_m } (e+|x|)^{\frac{2-m}{m+1}} \big(\ln(e+|x|)\big)^{\frac{1}{m+1}},\quad \forall\ x\in\Z^2.   
\end{align}

\end{proposition}

  For any $a\in \R,$ we use the notation  $\lceil a\rceil:=\min\{n\in \Z: n\geq a\}.$ \begin{lemma}\label{lm 2-1-1}
 Let $\sigma >2,$ and $\rho\in\R.$  Then for any $r\geq 4,$ 
\begin{align}\label{mm-00}
   \int_{  \Z^2  \setminus B_r}  |y|^{-\sigma} (\ln |y|)^{-\rho} dy \leq  \pi2^{2\sigma+2|\rho|-1}  \varpi_0(\sigma, \rho) r^{2-\sigma}  (\ln r)^{-\rho}, 
  \end{align}
 where 
 \begin{equation}\label{epp 1}
\varpi_0(\sigma,\rho)=
 \left\{
\begin{array}{lll}
 \frac{\lceil-\rho\rceil! \sigma^{\lceil-\rho\rceil} }{ (\sigma-2)^{\lceil-\rho\rceil+1} }, \quad
   &{\rm if}\ \  \sigma>2,\ \rho<0, \\[2mm]
  (\sigma-2)^{-1},\quad
   &{\rm if}\ \  \sigma>2,\ \rho\geq0. \\[2mm]
 \end{array}
 \right.
\end{equation}
For the case $\sigma>2$ and $\rho=-1,$ we have for any $r\geq 4e^{\frac{2}{\sigma-2}},$
$$\int_{  \Z^2  \setminus B_r}  |y|^{-\sigma} \ln |y| dy\leq \frac{\pi 2^{2\sigma+3}}{\sigma-2}r^{2-\sigma}\ln r.$$
 \end{lemma}
 \noindent{\bf Proof. }  
One easily shows that for $|x|\geq r\geq 4$ and any $y\in Q_{1}(x),$
\begin{equation}\label{eq:a1}
\frac12 \leq  \frac{|y|}{|x|}\leq 2,\qquad   \frac12  \leq  \frac{\ln |y|}{\ln |x|}\leq  2. 
 \end{equation}
Set $A:=\cup_{x\in \Z^2\setminus B_r}Q_{1}(x)$ with $r\geq 4$.  Define a function $h:A\to \R$ as

$$h(y)=\sum_{x\in \Z^2\setminus B_r}|x|^{-\sigma} (\ln |x|)^{-\rho} \mathds{1}_{Q_{1}(x)}(y),$$

where $\mathds{1}_{Q_{1}(x)}(\cdot)$ is the indicator function on $Q_{1}(x).$
Hence, by \eqref{eq:a1}
\begin{align*} 
0\leq \sum_{  x\in \Z^2  \setminus B_r}  |x|^{-\sigma} (\ln |x|)^{-\rho} =\int_{A}h(y)d\mathcal{L}^2(y)
 \leq 2^{\sigma+|\rho|}\int_{A}|y|^{-\sigma} (\ln |y|)^{-\rho}d\mathcal{L}^2(y),
\end{align*} where $\mathcal{L}^2$ is the Lebesgue measure on $\R^2.$
Since $\R^2\setminus B_{r+1}\subset A\subset \R^2\setminus B_{r-1},$
$$ \int_{A}|y|^{-\sigma} (\ln |y|)^{-\rho}d\mathcal{L}^2(y)\leq \int_{\R^2\setminus B_{r-1}}|y|^{-\sigma} (\ln |y|)^{-\rho}d\mathcal{L}^2(y).$$
The results are reduced to the estimates of 
$$\int_{\R^2\setminus B_{r}}|y|^{-\sigma} (\ln |y|)^{-\rho}d\mathcal{L}^2(y)=2\pi\int_{r}^\infty t^{1-\sigma} (\ln t)^{-\rho} dt.$$
 
Let $r>1.$ For the case that $\sigma>2$ and $\rho\geq 0,$  we have that 
$$\int_{r}^\infty t^{1-\sigma} (\ln t)^{-\rho} dt\leq (\ln r)^{-\rho}\int_{r}^\infty t^{1-\sigma}  dt=\frac{1}{\sigma-2}r^{2-\sigma}(\ln r)^{-\rho}.$$ 
 For the case that $\sigma>2$ and $\rho\in[-1,0),$ we derive that  
 \begin{align*}
\int_{r}^\infty t^{1-\sigma} (\ln t)^{-\rho} dt&= \frac{1}{\sigma-2}r^{2-\sigma}(\ln r)^{-\rho}-\frac{\rho}{\sigma-2}\int_{r}^\infty t^{1-\sigma} (\ln t)^{-\rho-1} dt 
\\&\leq  \frac{1}{\sigma-2}r^{2-\sigma}(\ln r)^{-\rho} +\frac{-\rho}{(\sigma-2)^2}r^{2-\sigma}(\ln r)^{-\rho-1} 
\\&\leq  \big(\frac{1}{\sigma-2} +\frac{-\rho}{(\sigma-2)^2}\big) r^{2-\sigma}(\ln r)^{-\rho}   
\end{align*}
For the case $\sigma>2$ and $\rho=-1,$ we have for $r\geq e^{\frac{2}{\sigma-2}}$
\begin{eqnarray*}
\int_{r}^\infty t^{1-\sigma} \ln t dt&=& \frac{1}{\sigma-2}r^{2-\sigma}\ln r+\frac{1}{\sigma-2}\int_{r}^\infty t^{1-\sigma} dt\\
&\leq& \frac{1}{\sigma-2}r^{2-\sigma}\ln r+\frac12\int_{r}^\infty t^{1-\sigma} \ln t dt.
\end{eqnarray*} This yields that
$$\int_{r}^\infty t^{1-\sigma} \ln t dt\leq \frac{2}{\sigma-2}r^{2-\sigma}\ln r.$$

We recall the notation for $b\geq 0,$ $b!=b\cdot(b-1)\cdots(b-[b]+1)$ for $b>1$ and $b!=1$ if $b\in[0,1].$
For the case $\sigma>2$ and $\rho\in[-n,1-n),$ $\lceil -\rho\rceil=n.$ Then we have
 \begin{align*}
&\quad\ \int_{r}^\infty t^{1-\sigma} (\ln t)^{-\rho} dt
\\&= \frac{1}{\sigma-2}r^{2-\sigma}(\ln r)^{-\rho}-\frac{\rho}{\sigma-2}\int_{r}^\infty t^{1-\sigma} (\ln t)^{-\rho-1} dt 
\\&=\sum^{n-1}_{m=0} \frac{ (-\rho)! }{(-\rho-m)!  (\sigma-2)^{m+1}} r^{2-\sigma}(\ln r)^{-\rho-m} +\frac{(-\rho)! }{(\sigma-2)^n}\int_{r}^\infty t^{1-\sigma} (\ln t)^{-\rho-n} dt  
\\&\leq  \sum^{n-1}_{m=0} \frac{ (-\rho)! }{(-\rho-m)! (\sigma-2)^{m+1}} r^{2-\sigma}(\ln r)^{-\rho-m} +\frac{(-\rho)! }{(\sigma-2)^{n+1}} r^{2-\sigma}(\ln r)^{-\rho-n}
\\&\leq  \big(\sum^{n}_{m=0} \frac{ n! }{(n-m)! (\sigma-2)^{m+1}} \big) r^{2-\sigma}(\ln r)^{-\rho} 
\\&\leq    \frac{n! \sigma^{n} }{ (\sigma-2)^{n+1} }r^{2-\sigma}(\ln r)^{-\rho}.
\end{align*}

Combining all above estimates, we prove the result. \hfill$\Box$\medskip 

 \begin{remark} 
 
By the same proof in the above proposition, 
there is a universal constant $C_1>0$
such that $$C_1^{-\sigma} \int_{\R^2}  (1+|x|)^{-\sigma} d\mathcal{L}^2(x)\leq\int_{\Z^2}  (1+|x|)^{-\sigma} dx\leq C_1^\sigma \int_{\R^2}  (1+|x|)^{-\sigma} d\mathcal{L}^2(x).$$ Hence,
there exists a universal constant $C>0$ such that for $\sigma>2$
 \begin{equation}\label{eq:general1} 
\frac{C^{-\sigma}}{\sigma-2}\leq \int_{\Z^2}  (1+|x|)^{-\sigma} dx\leq \frac{C^\sigma}{\sigma-2}.\end{equation}
 \end{remark}


 \noindent{\bf Proof of Proposition \ref{pr 2.1}. }  
 Without loss of generality, we assume $\|f\|_{\ell^\infty_m }=1.$ Set 
 $$r_0:=\left(e^{10m}+\frac{1}{m-2}\right)^{\frac{3}{m-2}}.$$ Let $x\in \Z^2.$ We divide into the following cases.
 
 \textbf{Case 1.} $|x|\leq r_0.$
 \begin{eqnarray*}|\Phi_0\ast f(x)|&\leq& C \int_{y\in\Z^2}\ln(1+|y-x|)(1+|y|)^{-m}dy\\
 &\leq &C \left(\int_{\Z^2\cap B_{2r_0}}\ln(4r_0)(1+|y|)^{-m}dy+\int_{\Z^2\setminus B_{2r_0}}\ln(3|y|)|y|^{-m}dy\right).
 \end{eqnarray*}
  By \eqref{eq:general1} and Lemma~\ref{lm 2-1-1}, there exists a constant $C>1$ such that
 $$|\Phi_0\ast f(x)|\leq\frac{C^m}{m-2}\ln r_0.$$ Hence for any $|x|\leq r_0,$
 \begin{eqnarray*}|\Phi_0\ast f(x)|&\leq& \frac{C^m}{m-2}\ln r_0(e+r_0)^{\frac{m-2}{m+1}}(e+|x|)^{\frac{2-m}{m+1}} \big(\ln(e+|x|)\big)^{\frac{1}{m+1}}.\end{eqnarray*}
 The result follows from choosing large $\widetilde{C}$ such that $\frac{C^m}{m-2}\ln r_0(e+r_0)^{\frac{m-2}{m+1}}\leq {\frac{\widetilde{C}^m}{(m-2)^4}}.$
 
  
  \textbf{Case 2.} $|x|> r_0.$ Set 
  $$R:= |x|^{\theta_1}\big(\ln(e+|x|)\big)^{\theta_2}$$
   with 
   $$\theta_1=\frac{3}{m+1},\quad \theta_2=\frac{1}{m+1}.$$
 First, we show that \begin{equation}\label{eq:cc1}\frac{\ln|x|}{|x|^{m-2}}\leq \frac{1}{20^{m+1}}\leq 1.\end{equation} Since $|x|\mapsto \frac{\ln|x|}{|x|^{m-2}}$ is non-increasing for $|x|>r_0,$ it suffices to prove that
 $$\frac{\ln r_0}{r_0^{m-2}}\leq e^{-4(m+1)},$$
which is equivalent to
 $$\ln \ln r_0-(m-2) \ln r_0+4(m+1)\leq 0.$$ In fact,
 \begin{align*}
 \ln \ln r_0-(m-2) \ln r_0+4(m+1)&\leq \ln 3-\ln(e^{10m}+\frac{1}{m-2})+4(m+1)
 \\&\leq \ln 3+4-6m
  \\& \leq 0.
 \end{align*}
 This yields the result.


Moreover, we prove that $R\leq \frac{1}{10}|x|.$ In fact, 
$$\frac{R}{|x|}\leq 2|x|^{\theta_1-1}(\ln|x|)^{\theta_2}=2\left(\frac{\ln|x|}{|x|^{m-2}}\right)^{\frac{1}{m+1}}\leq\frac{1}{10}.$$
 
 Next, we write
 \begin{eqnarray*}\Phi_0\ast f(x)&=&\left(\int_{\Z^2\cap B_R}+\int_{\Z^2\cap (B_{4|x|}(x)\setminus B_R)}+\int_{\Z^2\setminus B_{4|x|}(x)}\right) \Phi_0(x-y)f(y)dy\\
 &=:&I+II+III.
 \end{eqnarray*}

For the term $I,$ 
 \begin{eqnarray*} I&=&\int_{\Z^2\cap B_R}\left(\Phi_0(x-y)-\Phi_0(x)\right)f(y)dy+\int_{\Z^2\cap B_R}\Phi_0(x)f(y)dy\\
 &=:&I_1+I_2.
 \end{eqnarray*}
 For the term $I_1,$ by \eqref{fund-1}, there is a constant $C$
 such that
 $$\left|\Phi_0(x)+\frac1{2\pi}\ln |x|+\frac{\gamma_0}{2}\right|\leq \frac{C}{1+|x|},\quad \forall x\in \Z^2.$$ Hence, by $R\leq \frac{1}{10}|x|,$ for any $y\in \Z^2\cap B_R,$ 
 $$|\Phi_0(x-y)-\Phi_0(x)|\leq \frac1{2\pi}|\ln|x-y|-\ln |x||+\frac{C}{1+|x-y|}+\frac{C}{1+|x|}\leq C\left(|\ln|x-y|-\ln |x||+\frac{1}{|x|}\right).$$
By the mean value theorem, for some $\xi$ between $|x|$ and $|x-y|,$ 
$$|\ln|x-y|-\ln |x||=\frac{1}{\xi}||x-y|-|x||\leq \frac{10}{9}\frac{R}{|x|}.$$
 This implies that
 $$|I_1|\leq C\frac{R}{|x|}\int_{\Z^2\cap B_R}|f(y)|dy\leq \frac{C^m}{m-2}\frac{R}{|x|}.$$
 
 For the term $I_2,$ by $\int_{\Z^2}f=0$ and (\ref{mm-00}) with $\rho=0$,  \begin{eqnarray*}|I_2|&=&|\Phi_0(x)|\left|\int_{\Z^2\cap B_R}f(y)dy\right|=|\Phi_0(x)|\left|\int_{\Z^2\setminus B_R}f(y)dy\right|\\
 &\leq& C\ln|x|  \int_{\Z^2\setminus B_R}(1+|y|)^{-m}dy\leq \frac{C^m}{m-2}\ln|x| R^{2-m}.\end{eqnarray*}
 

 
 Now we consider the term $II,$
 \begin{eqnarray*}
 |II|&\leq&  R^{-m}\int_{\Z^2\cap (B_{4|x|}(x)\setminus B_R)} |\Phi_0(x-y)|dy\leq   R^{-m}\int_{\Z^2\cap B_{4|x|}} |\Phi_0(y)|dy\\
 &\leq& C  R^{-m}|x|^2\ln |x|.
 \end{eqnarray*}
 
 For the term $III,$ noting that $|y|\geq 3|x|$ for any $y\in \Z^2\setminus B_{4|x|}(x),$ we have
 \begin{eqnarray*}
 |III|&\leq&\int_{\Z^2\setminus B_{4|x|}(x)} |\Phi_0(x-y)f(y)|dy\leq C\int_{\Z^2\setminus B_{|x|}} \ln(3|y|) |y|^{-m}dy\\
 &\leq&\frac{C^m}{m-2}|x|^{2-m} \ln|x|.
 \end{eqnarray*}
 Combining all above estimates,
 we have
 \begin{eqnarray*}|\Phi_0\ast f(x)|&\leq& \frac{C^m}{m-2}\left(\frac{R}{|x|}+\ln|x| R^{2-m}+R^{-m}|x|^2\ln |x|+|x|^{2-m}\ln |x|\right)\\
 &\leq&\frac{C^m}{m-2}|x|^{\theta_1-1}(\ln|x|)^{\theta_2}\left(2+\left(\frac{\ln|x|}{|x|^{m-2}}\right)^{2\theta_2}+\left(\frac{\ln|x|}{|x|^{m-2}}\right)^{m\theta_2}\right)\\
 &\leq&\frac{C^m}{m-2}|x|^{\theta_1-1}(\ln|x|)^{\theta_2},\end{eqnarray*}
 where we have used \eqref{eq:cc1}. 
 
 This proves the result.
   \hfill$\Box$
 \medskip

The following are corollaries of Proposition \ref{pr 2.1}. 
\begin{corollary}\label{cr 2.1}
Let $f\in \ell^\infty_m(\Z^2)$ with $m>2$ satisfy (\ref{sst-1}). 
  Then for any $\tau\in(0,\frac{m-2}{m+1}),$ $\Phi_0\ast f\in \ell^\infty_{\tau}(\Z^2)$  and 
 \begin{align}\label{sst-3}
\|\Phi_0\ast f\|_{\ell^\infty_\tau } \leq b_{m,\tau} \|f\|_{\ell^\infty_m }.    
\end{align}
where 
$$b_{m,\tau}=\frac{c_0^m }{(m-2)^4}  \big(\frac1{ m-2-\tau(m+1)} \big)^{\frac1{m+1}} $$
and $c_0>0$ is given in Proposition \ref{pr 2.1}.
\end{corollary}
 \noindent{\bf Proof. } By Proposition \ref{pr 2.1}, we have that 
 \begin{align*}
\|\Phi_0\ast f\|_{\ell^\infty_\tau } & \leq \sup_{x\in\Z^2}\Big(\big|\big(\Phi_0\ast f\big)(x) \big| (e+|x|)^{\tau}\Big)  \\ &\leq \frac{c_0^m}{(m-2)^4}\|f\|_{\ell^\infty_m }(e+|x|)^{\frac{2-m}{m+1}+\tau} \big(\ln(e+|x|)\big)^{\frac{1}{m+1}}
\\&\leq \frac{c_0^m}{(m-2)^4} \|f\|_{\ell^\infty_m } \, \sup_{x\in\Z^2}\Big( (e+|x|)^{\frac{2-m}{m+1}+\tau} \big(\ln(e+|x|)\big)^{\frac{1}{m+1}}\Big). 
\end{align*}
 For $0<\tau<\frac{m-2}{m+1}$, direct calculation shows that 
 \begin{align*}
 \sup_{x\in\Z^2} (e+|x|)^{\frac{2-m}{m+1}+\tau} \big(\ln(e+|x|)\big)^{\frac{1}{m+1}}&\leq e^{-\frac{1}{m+1}}\big(\frac1{ m-2-\tau(m+1)} \big)^{\frac1{m+1}}
  \\& \leq\big(\frac1{ m-2-\tau(m+1)} \big)^{\frac1{m+1}}.\end{align*}
This proves the result.   \hfill$\Box$

 \begin{corollary}\label{cr 2.2}
Let $f\in \ell^\infty_{m}(\Z^2)$ with $m>2$ satisfy that 
\begin{align}\label{sst-1-1}
\beta_f:=\frac{1}{2\pi} \int_{\Z^2} f(x)dx >0. 
\end{align}  
Then there exists a universal constant $C>1$ such that
 \begin{align}\label{sst-3-0}
 \Big|\big(\Phi_0\ast f\big)(x) + \beta_f \ln(1+|x|)+\pi\gamma_0\beta_f\Big|\leq \frac{C^m}{(m-2)^4} \big(\|f\|_{\ell^\infty_m }+\beta_f\big) (e+|x|)^{\frac{2-m}{m+1}}\big( \ln(e+|x|)\big)^{\frac{1}{m+1}}.
\end{align}
\end{corollary} 
\noindent{\bf Proof. }  Let 
$$f_1=f-\big(\int_{\Z^2} f(x)dx\big) \delta_0.$$ 
Then $\int_{\Z^2} f_1 dx=0$, $f_1\in \ell^\infty_{m}(\Z^2)$ with $m>2$ and we apply Proposition \ref{pr 2.1} to obtain
that 
\begin{align*}
\big|\big(\Phi_0\ast f_1\big)(x)   \big|&\leq \frac{c_0^m}{(m-2)^4}\|f_1\|_{\ell^\infty_m }(1+|x|)^{\frac{2-m}{m+1}} \big(\ln(e+|x|)\big)^{\frac{1}{m+1}}
\\&\leq \frac{c_0^m}{(m-2)^4}\left(\|f\|_{\ell^\infty_m }+\int_{\Z^2} f\right)(1+|x|)^{\frac{2-m}{m+1}}\big(\ln(e+|x|)\big)^{\frac{1}{m+1}} 
\end{align*}
and 
$$\big(\Phi_0\ast f_1\big)(x)  =\big(\Phi_0\ast f\big)(x) - \left(\int_{\Z^2} f\right)\,  \Phi_0(x) \quad{\rm for}\ \, x\in\Z^2,$$
which, together with (\ref{fund-1}),  implies (\ref{sst-3-0}). 
 \hfill$\Box$

 \setcounter{equation}{0}
 \section{Problem with source nonlinearity}
   
 \subsection{Bounded solution for modified model}
  
 Next   we  give  the existence result for the following   modified equation 
\begin{equation}\label{eq 2.2-IIs}
- \Delta u= K e^{\kappa u}-  g\quad
    {\rm in}\ \  \Z^2,
\end{equation}
where $K,g: \Z^2\to [0,\infty).$   We recall the well-known Schauder's fixed point theorem.    
\begin{theorem}\label{fixed-point}{\cite[Corollary 11.2]{GT}}
Let $D$ be a bounded convex subset of a Banach space $X$ with $0\in D.$ If $A:D\to D$ is a continuous and compact mapping. Then $A$ has a fixed point in $D.$
\end{theorem}
For a function $f:\Z^2\to \R,$ we write $f\gneqq 0$ if $f\geq 0$ and $f\not\equiv 0$ on $\Z^2.$

   \begin{proposition}\label{cr 3.1}
Assume $K,g\in \ell^\infty_{\tau_0}(\Z^2)$  with $\tau_0>2$,   $K,g\gneqq 0,$  and satisfy
  \begin{align}\label{req-1--12}
   \|   K\|_{\ell^\infty_{\tau_0}}  \leq  \frac{1 }{  2b_{  \tau_0 }  }\frac{r_0}{ e^{ \tilde c_0+2\kappa r_0 }},\qquad  \|g\|_{ \ell^\infty_{\tau_0}} \leq   \frac{  1}{2b_{  \tau_0 } }r_0
 \end{align}
 for  some  $r_0>0$, where $ \tilde c_0= \ln\Big(\frac{\int_{\Z^2}g(x)dx}{\int_{\Z^2}K(x)dx}\Big)$,  $b_{  \tau_0 }=b_{  \tau_0,\tau_2}$ is given in Corollary \ref{cr 2.1} with 
 $$m=\tau_0,\qquad  \tau=\tau_2:=\frac23 \frac{\tau_0-2}{\tau_0+1}. $$
  
 Then  the problem (\ref{eq 2.2-IIs})  has  a  solution $u_0$   such that    
  \begin{equation}\label{asymp-1}
  u_0(x)= u_{\infty}+O(|x|^{\frac{2-\tau_0}{\tau_0+1}}(\ln |x|)^{\frac{1 }{\tau_0+1}} )\quad {\rm as}\ \  x\to\infty, 
  \end{equation}
where   $u_{\infty}\in\R$ satisfies that 
   \begin{equation}\label{asymp-1-1}
     \Big| u_{\infty}-\frac1\kappa\ln\big(\frac{\int_{\Z^2}g(x)dx}{\int_{\Z^2}K(x)dx}\big)\Big|\leq  r_0.    \end{equation} 
     
    Furthermore,  we have that 
      \begin{equation}\label{asymp-1+1}
     \int_{\Z^2}K e^{\kappa u_0}dx = \int_{\Z^2} g dx.   
       \end{equation} 
  
 \end{proposition}
 \noindent{\bf Proof. }   
     Let $\tau_1=\frac{1}{2} \frac{  \tau_0-2}{ \tau_0+1}>0.$ For $v\in \ell^\infty_{\tau_1}(\Z^2),$ define 
    $$\cT_0(v)=  \Phi_0\ast \big(   K e^{\kappa  v+c_v}-g\big),$$
    where
     $$c_v := \ln\big(\frac{\int_{\Z^2}g(x)dx}{\int_{\Z^2}K(x)e^{\kappa v(x)}dx}\big).  $$

Set $$\bX_{r_0}=\Big\{v\in \ell^\infty_{\tau_1}(\Z^2):\,  \|v\|_{\ell^\infty_{\tau_1}}\leq  r_0 \Big\}.$$ Note that $\bX_{r_0}$ is a bounded convex subset of $\ell^\infty_{\tau_1}(\Z^2).$
Note that for $v\in \bX_{r_0}$,  direct computation shows that 
 \begin{align}\label{req--0}
  \tilde  c_0 -  r_0\kappa  \leq c_v\leq   \tilde  c_0    + r_0\kappa,  
 \end{align} 
 where $\tilde  c_0:=\frac{\int_{\Z^2}g(x)dx}{\int_{\Z^2}K(x)dx}.$

For $\tau_2=\frac23\, \frac{  \tau_0-2}{ \tau_0+1}$ and $u\in \bX_{r_0},$ by Corollary \ref{cr 2.1},
\begin{align}\label{est-1}
\|\cT_0(u)\|_{\ell^\infty_{\tau_1}} \leq   \|\cT_0(u)\|_{\ell^\infty_{\tau_2}} &\leq \big\|  \Phi_0\ast \big(   K e^{\kappa  u+c_u}-g\big)   \big\|_{\ell^\infty_{\tau_2}}
\\&\leq   b_{  \tau_0,\tau_2}   \big\|   \big(   K e^{\kappa  u+c_u}-g\big)   \big\|_{\ell^\infty_{  \tau_0}}
\nonumber\\ &\leq    b_{  \tau_0}     \Big(   e^{ \tilde c_0+2\kappa r_0}   \|   K\|_{\ell^\infty_{  \tau_0}} + \| g\|_{\ell^\infty_{  \tau_0}}\Big)
\nonumber\\ &\leq  r_0,\nonumber
\end{align}
where the last inequality follows from \eqref{req-1--12}.
Hence, we have the mapping $$\cT_0:\bX_{r_0} \to \bX_{r_0}.$$
By Lemma \ref{lm 2.1}, $\ell^\infty_{\tau_2}(\Z^2)\hookrightarrow \ell^\infty_{\tau_1}(\Z^2)$ is compact, and hence $\cT_0$ is a compact mapping by \eqref{est-1}. 

Next, we show that $\cT_0$ is continuous. Let $\{v_n\}_{n=1}^\infty\cup \{v\}\subset \bX_{r_0}$ such that $v_n\to v$ in $\ell^\infty_{\tau_1}$ as $n\to \infty.$ For $\tau_1>0,$
$$\sup_n\|v_n\|_{\ell^\infty}+\|v\|_{\ell^\infty}\leq \sup_n\|v_n\|_{\ell^\infty_{\tau_1}}+\|v\|_{\ell^\infty_{\tau_1}}<\infty$$ and $\|v_n-v\|_{\ell^\infty}\to 0$  as $n\to \infty.$ By the dominated convergence theorem, $c_{v_n}\to c_v$ as $n\to \infty$
and 
 $e^{ c_{v_n}} \to e^{  c_v}$ as $n\to \infty$

 Moreover,
\begin{align*}
\|\cT_0(v_n)-\cT_0(v)\|_{\ell^\infty_{\tau_1}}&\leq b_{\tau_0} \|K e^{\kappa  v_n+c_{v_n}}-K e^{\kappa  v+c_{v}}\|_{\ell^\infty_{\tau_0}}\\
&\leq b_{\tau_0} \|K\|_{\ell^\infty_{\tau_0}}\|e^{\kappa  v_n+c_{v_n}}- e^{\kappa  v+c_{v}}\|_{\ell^\infty}
\\
&\leq b_{\tau_0} \|K\|_{\ell^\infty_{\tau_0}}\Big( \big|e^{c_{v_n}}- e^{c_{v}}\big|e^{\kappa  \|  v_n\|_{\ell^\infty}}  +e^{c_{v}}  e^{\kappa \|   {v_n}-    {v}\|_{\ell^\infty} }  \Big)
\\
&\to 0\quad{\rm as}\ \, n\to\infty.
\end{align*} 
This proves the continuity of $\cT_0.$
 
 By Schauder's fixed point theorem, Theorem~\ref{fixed-point}, we obtain a fixed point $v_0\in \bX_{r_0},$
 $$v_0=  \Phi_0\ast \big(   K e^{\kappa  v_0 +c_{v_0}}-g\big) $$
 and setting $u_0=  v_0 +\frac1\kappa c_{v_0}  $, 
 $$-\Delta u_0= K e^{\kappa  u_0}-  g \quad   {\rm in}\ \  \Z^2.  $$
This proves the existence of the solution to \eqref{eq 2.2-IIs}.  

The estimates of  (\ref{asymp-1}), (\ref{asymp-1-1}) and (\ref{asymp-1+1}) follow from Proposition~\ref{pr 2.1}, (\ref{est-1}) and 
the choice of $c_v$ in definition of $\cT_0$ . \hfill$\Box$ \medskip

\subsection{Source Case} 
We would apply Proposition \ref{cr 3.1} to obtain the solution of (\ref{eq 1.1}). \medskip


 
\noindent {\bf Proof of Theorem \ref{teo 1}. } Let $\kappa>0, \alpha>\beta\geq 0.$ We need to find a solution 
$u_\alpha$ having the asymptotics $u_\alpha\sim -\frac{\alpha}{2\pi}\ln(|x|) $
as $x\to\infty.$ We write
$$u:=\tilde u+  \alpha \Phi_{0}\quad   \text{ in $\Z^2$}.$$
It suffices to find a solution $\tilde u$ of 
 \begin{align}\label{eq:p12}
  - \Delta \tilde u=  K_{\alpha}  e^{\kappa \tilde u}- g_{\alpha,\beta} \quad
    {\rm in}\ \  \Z^2,
 \end{align}
where 
$$K_{\alpha} (x)=e^{ \alpha\kappa \Phi_{0}(x)},\qquad  g_{\alpha,\beta} =(\alpha -\beta)\delta_0. 
$$ 

In the following, we write $\sigma:=\frac{\alpha \kappa}{2\pi}.$ We always assume $\sigma>2.$ Instead of the pair of parameters $(\kappa,\alpha),$ we consider the pair $(\kappa, \sigma).$
By \eqref{eq:q1}, 
  \begin{align*} 
 e^{-2\pi c_1\sigma}   (1+|x|)^{-\sigma}\leq K_{\alpha}  (x) \leq  e^{2\pi c_1\sigma}  (1+|x|)^{-\sigma}.
 \end{align*} Hence, by \eqref{eq:general1}, there exists $C_2>1$ such that
  \begin{align*} 
  \frac{C_2^{-\sigma}}{\sigma-2}   \leq \int_{\Z^2} K_{\alpha}  (x)dx  \leq \frac{C_2^\sigma}{\sigma-2}  .
 \end{align*}
For
 $$\tilde c_{\alpha}=\ln\Big(\frac{\int_{\Z^2}g_{\alpha,\beta}(x)dx}{\int_{\Z^2}K_{\alpha} (x)dx}\Big),$$  
we have 
 $$e^{\tilde c_{\alpha}}\leq(\alpha-\beta)(\sigma-2)C_2^\sigma.$$

 By applying Proposition \ref{cr 3.1} with the setting 
   $$\tau_0=\sigma,\quad   K=K_{\alpha}\quad{\rm and}\quad  g=g_{\alpha,\beta},$$
  we need to verify  (\ref{req-1--12}), i.e.
  \begin{align}\label{req-1--012}
   \|   K_{\alpha}  \|_{\ell^\infty_{\sigma}}  \leq  \frac{1 }{2\kappa b_{\sigma} } \frac{r_0}{ e^{ \tilde c_\alpha+2\kappa r_0 }},\qquad  \|g_{\alpha,\beta} \|_{\ell^\infty_{\sigma}} \leq  \frac{ r_0}{2b_{\sigma}} 
 \end{align}
 for free parameter $r_0>0$. 
By choosing $r_0=2b_{\sigma} (\alpha-\beta),$ we only need to show that
\begin{equation}\label{eq:q2}
\|   K_{\alpha}  \|_{\ell^\infty_{\sigma}}  \leq\frac{\alpha-\beta}{\kappa}e^{-\tilde c_\alpha-4\kappa b_{\sigma}(\alpha-\beta)}.
\end{equation}
Note that for $\tau_2=\frac{2}{3}\frac{\sigma-2}{\sigma+1},$
\begin{align*} 
b_\sigma&= b_{\sigma,\tau_2}=\frac{c_0^\sigma}{(\sigma-2)^4}\left(\frac{3}{\sigma-2}\right)^{\frac{1}{\sigma+1}} \leq 3c_0^\sigma(\sigma-2)^{-4-(\sigma+1)^{-1}}.
\end{align*} 
Hence, by $\beta\geq 0,$
\begin{eqnarray*}\frac{\alpha-\beta}{\kappa}e^{-\tilde c_\alpha-4\kappa b_{\sigma}(\alpha-\beta)}&\geq& \frac{C_2^{-\sigma}}{\kappa(\sigma-2)}e^{-12c_0^\sigma\kappa (\alpha-\beta)(\sigma-2)^{-4-(\sigma+1)^{-1}}}\\
&\geq&\frac{C_2^{-\sigma}}{\kappa(\sigma-2)}e^{-24\pi c_0^\sigma \sigma (\sigma-2)^{-4-(\sigma+1)^{-1}}}\end{eqnarray*}
Since $\|   K_{\alpha}  \|_{\ell^\infty_{\sigma}}\leq e^{2\pi c_1\sigma},$ for \eqref{eq:q2},
it suffices to prove that
\begin{equation}\label{eq:qq1}\frac{1}{\kappa}\geq C_2^{\sigma}(\sigma-2)e^{2\pi c_1\sigma+24\pi c_0^\sigma \sigma (\sigma-2)^{-4-(\sigma+1)^{-1}}}=:h_0(\sigma).\end{equation}
Hence, for any fixed $\sigma>2,$ there exists sufficiently small $\kappa$ satisfying \eqref{req-1--012}, and we have a solution for $\beta<\alpha.$
Note that 
$$\lim_{\sigma\to 2}h_0(\sigma)=\lim_{\sigma\to\infty}h_0(\sigma)=\infty.$$
 In fact, we have the following.
\begin{enumerate}
\item[$(i)$] Let $\kappa^*:=\big(\min_{t\in(2,\infty)}h_0(t)\big)^{-1}$ and $a_0=\mathrm{argmin}_{(2,\infty)}h.$ For any $\kappa\in(0, \kappa^*],$ \eqref{eq:qq1} holds for $\sigma=a_0.$ In this case, $\alpha=\frac{2\pi a_0}{\kappa}.$  For any $\beta\in(0,\alpha],$ we have a solution $\tilde{u}$ to \eqref{eq:p12}, and hence a solution $u$ to \eqref{eq 1.1}.
\item[$(ii)$] For any $\epsilon\in(0,1),$ let 
$$\overline{\kappa}(\epsilon):=\big(\max_{t\in [2+\epsilon,2+\epsilon^{-1}]}h(t)\big)^{-1}\in (\kappa^*,+\infty).$$ Then for any $\kappa\in(0,\overline{\kappa}(\epsilon)]$ and any $\sigma\in[2+\epsilon,2+\epsilon^{-1}],$ we have \eqref{eq:qq1}. In this case,
$$\alpha=\frac{2\pi \sigma}{\kappa}\in [\frac{2\pi (2+\epsilon)}{\kappa},\frac{2\pi (2+\epsilon^{-1})}{\kappa}].$$ For any $\beta\in(0,\alpha],$ we have a solution  $\tilde{u}$ to \eqref{eq:p12}, and hence a solution  $u$ to \eqref{eq 1.1}.

\end{enumerate}

 Next, we estimate the solutions. Without loss of generality, we consider the second case in the above. By
Proposition \ref{cr 3.1}, the solution $\tilde{u}$ to \eqref{eq:p12} satisfies
$$\tilde{u}(x)= c+O(|x|^{\frac{4\pi- \kappa \alpha}{\kappa \alpha +2\pi}}\big(\ln |x|\big)^{\frac{2\pi }{\kappa \alpha+2\pi}}) \quad {\rm as}\ \ |x|\to+\infty,  $$
where
  \begin{align}\label{ess-1--0}
  \Big| c-\frac1\kappa \ln\big(\frac{\int_{\Z^2}g_{\alpha,\beta}(x)dx}{\int_{\Z^2}K_{\alpha}(x)dx}\big)\Big|\leq r_0. 
   \end{align}
For the solution of (\ref{eq 1.1}), $u=\tilde{u}+ \alpha \Phi_{0}$,  we have
$$
u(x)=-\frac{\alpha}{2\pi}\ln|x|+d_{\kappa,\alpha,\beta}+O(|x|^{\frac{4\pi-\alpha\kappa}{\alpha\kappa+2\pi}}\big(\ln |x|\big)^{\frac{2\pi }{\alpha\kappa+2\pi}}) 
\quad {\rm as}\ \,  |x|\to+\infty,
$$
where 
 \begin{align*} 
 d_{\kappa,\alpha,\beta}=c -\alpha\frac{\gamma_0}{2}.   
 \end{align*}
 Moreover,    from (\ref{asymp-1+1}) 
 \begin{align*}\label{ep-1-e}
\int_{\Z^2} e^{\kappa u} dx=\int_{\Z^2} K_{\alpha}  e^{\kappa \tilde u}dx = \int_{\Z^2} g_{\alpha,\beta} dx=\alpha-\beta. 
\end{align*}
This proves the results. \hfill$\Box$\medskip

\noindent{\bf Proof of Theorem~\ref{thm:main0}.}   By Theorem~\ref{teo 1}, for any $\epsilon\in(0,1)$, there exists $\overline{\kappa}=\overline{\kappa}(\epsilon)$ such that for $\kappa=\overline{\kappa},$ and any $\alpha\in[\frac{2\pi(2+\epsilon)}{\kappa},\frac{2\pi(2+\epsilon^{-1})}{\kappa}],$
there exists a solution $u_{\kappa,\alpha,0}$ to the Liouville equation \eqref{eq 1.1-0} satisfying 
\begin{align*}
u_{\kappa,\alpha,0}(x)=-\frac{\alpha}{2\pi}\ln|x|+d_{\kappa,\alpha,0}+O(|x|^{\frac{4\pi-\alpha\kappa}{\alpha\kappa+2\pi}}\big(\ln |x|\big)^{\frac{2\pi }{\alpha\kappa+2\pi}}) 
\quad {\rm as}\ \,  x\to\infty,
\end{align*}
where $d_{\kappa,\alpha,0}\in\R$ depends on $\kappa,\alpha$ such that $\int_{\Z^2} e^{\kappa u_{\kappa,\alpha,0}} =\alpha<\infty.$

For any $a>4\pi,$ there exists $\epsilon\in(0,1)$ such that
$a\in[2\pi(2+\epsilon),2\pi(2+\epsilon^{-1})].$ For this $\epsilon,$
choose $\kappa=\overline{\kappa}(\epsilon),$ and $\alpha=\frac{a}{\kappa}.$ 

Set $u_a:=\kappa u_{\kappa,\alpha,0}+\log \kappa.$ 
One easily checks that $u_a$ satisfies $-\Delta u_a=e^{u_a}.$ Other assertions follow directly, with $c_a=\kappa d_{\kappa,\alpha,0}+\log\kappa.$
 \hfill$\Box$\medskip

       \setcounter{equation}{0}
 \section{Absorption case}
 
 \subsection{Bounded solutions of the modified model}

In this  subsection, we try to find the bounded solution of 
\begin{equation}\label{eq 2.2-ab}
- \Delta u+ K e^{\kappa  u}=g\quad
    {\rm in}\ \  \Z^2,
\end{equation}
where $K,g:\Z^2\to \R$ with $K\geq 0.$  

Here, we would like to note that \cite{N82,CN91} established the existence and uniqueness of the classical solution to the model  
\begin{equation}\label{eq 2.2-kkkk}  
- \Delta u + K e^{u} = 0 \quad \text{in } \mathbb{R}^2  
\end{equation}  
under the assumption that $ K(x) \sim |x|^{-l} $ near infinity, where $ l > 2 $. Further results on Eq.(\ref{eq 2.2-kkkk}) can be found in \cite{CL97,L00} and the references therein.

In order to find a solution of (\ref{eq 2.2-ab}),  we need to use the well-known Schaefer's fixed point theorem.     

\begin{theorem}\label{Schaefer fixed-point}{\cite[Theorem 4 in Chapter 9.2]{SE}}
Let  $\bX$  be a  Banach space $\bX$ and  $\mathcal{A}:\bX\to \bX$ be a continuous and compact mapping. 
 Assume further that the set  
 $$\big\{u\in \bX: u=t \mathcal{A}(u)\, \text{ for   $t\in[0,1]$}\big\}$$
 is bounded. 
 Then $\mathcal{A}$ has a fixed point in $\bX.$
\end{theorem}

   \begin{proposition}\label{pr 3.1-ab}
Let    $K,g\in \ell^\infty_{\tau_0}(\Z^2)$ with $\tau_0>2$,   $K\gneqq 0$,  and $\int_{\Z^2 } g(x)dx >0.$ Suppose that there exists $C_0$ such that 
\begin{equation}\label{con 0-ab}
g(x)\leq C_0 K(x),\quad\forall\,  x\in \Z^2,
\end{equation}
 then
  the problem (\ref{eq 2.2-ab})  admits   a  unique bounded  solution $w_0$, which has the asymptotic behavior
   $$w_0(x)=w_\infty+O\big(|x|^{\frac{2-\tau_0}{\tau_0+1}} (\ln |x|)^{\frac{1}{\tau_0+1}}\big)\quad {\rm as}\ \, |x|\to+\infty,  $$
where $w_\infty \in\R.$ 

Moreover, we have that 
$$ \int_{\Z^2}  K e^{\kappa w_0 }dx= \int_{\Z^2} g  dx,\qquad w_\infty\leq \frac1{\kappa} \ln \big(\| g  K^{-1}\|_{\ell^\infty(\mathrm{supp} K)}\big).$$

 \end{proposition}
\noindent{\bf Proof. } {\it Existence:}  Without loss of generality, we may assume $g\not\equiv K$ in $\Z^2$, otherwise, $u\equiv 0$ is the solution.

Set $\tau_1=\frac12 \frac{\tau_0-2}{\tau_0+1},\tau_2=\frac23 \frac{\tau_0-2}{\tau_0+1}.$
For any $v\in \ell^\infty_{\tau_1}(\Z^2),$ define
$$\cT_1(v)=\Phi_0\ast\big(g-K e^{\kappa  v+c_v }\big),$$
where 
$$c_v=  \ln\Big(\frac{\int_{\Z^2}g(x)dx}{\int_{\Z^2}K(x)e^{\kappa v(x)}dx}\Big),$$
which is finite by the assumptions that  $\int_{\Z^2 } g(x)dx>0$,  $K\gneqq 0,$ and $K,g\in\ell^\infty_{\tau_0}(\Z^2).$ \smallskip

Next we need to show $\cT_1$ has a fixed point.\smallskip 

{\it Step 1.} We show that  $\cT_1: \ell^\infty_{\tau_1}(\Z^2)\to \ell^\infty_{\tau_1}(\Z^2)$ and it is continuous and compact. 
These follow from the same arguments as in the proof of Proposition~\ref{cr 3.1}. Note that the compactness follows from the norm estimate of $\|\mathcal{T}_0(v)\|_{\ell_{\tau_2}^\infty(\Z^2)}$ and
 $$\cT_1: \ell^\infty_{\tau_1}(\Z^2)\to \ell^\infty_{\tau_2}(\Z^2)\subset \ell^\infty_{\tau_1}(\Z^2),$$
 where the latter embedding is compact by Lemma \ref{lm 2.1}.  \smallskip

{\it Step 2.} We prove that  
$${\bf A}:=\Big\{v\in \ell^\infty_{\tau_1}(\Z^2): v=t\cT_1(v)\, \text{ for  $t\in[0,1]$}\Big\}\ \ \text{ is bounded.}$$  For any $v\in {\bf A},$ 
there exists $t\in [0,1]$ such that
$v=t\cT_1(v),$ which implies that
\begin{align}\label{eq 2.3-ab}
-\Delta v=t\big(g-K e^{\kappa  v+c_{v}}\big)\quad {\rm in}\ \, \Z^2.
\end{align}
Note that (\ref{sst-3})
\begin{align}
\|v\|_{\ell^\infty_{\tau_1} } &=t \|\Phi_0\ast\big(g-K e^{\kappa  v+c_{v}}\big) \|_{\ell^\infty_{\tau_1} } \nonumber
 \\& \leq b_{\tau_1,\tau_0} \, \|\big(g-K e^{\kappa  v+c_{v}}\big) \|_{\ell^\infty_{\tau_0} } \nonumber
 \\& \leq b_{\tau_1,\tau_0}  \, \Big(\| g\|_{\ell^\infty_{\tau_0} }+ \|K\|_{\ell^\infty_{\tau_0}} \|e^{\kappa  v+c_{v}} \|_{\ell^\infty} \Big). \label{bound-1-ab}
 \end{align}  
 
 It suffices to prove that $\displaystyle\sup_{v\in {\bf A}}\|e^{\kappa  v+c_{v}} \|_{\ell^\infty}<\infty.$ Without loss of generality, we may assume $v\not\equiv 0$ and $t\in (0,1].$ For $v\in \ell^\infty_{\tau_1}(\Z^2)$,   $v(x)\to 0$ as $|x|\to \infty.$
We divide it into cases.\smallskip

{\it Case 1.} There exists $x_0\in \Z^2$ such that $v(x_0)\geq 0.$ Hence the maximum of $v$ is attained at some vertex $x_1.$ Since $v\not\equiv 0,$ $\displaystyle {\bf B}:=\{x\in \Z^2:v(x)=\sup_{x\in \Z^2}v\}\neq \Z^2.$ Hence there exists $x_2\in {\bf B}$ such that there is a neighbor of $x_2$ which is not in ${\bf B}.$ This implies that $\Delta v(x_2)<0.$ By  the equation (\ref{eq 2.3-ab}), 
\begin{align}  \label{bound-2-ab}
K(x_2) e^{\kappa v(x_2)+c_{v}}< g(x_2),
\end{align} which by \eqref{con 0-ab} implies $K(x_2)>0.$
Hence, 
$$\|e^{\kappa  v+c_{v}} \|_{\ell^\infty}\leq e^{\kappa v(x_2)+c_{v}}\leq \| g  K^{-1}\|_{\ell^\infty(\mathrm{supp} K)}.$$
By letting $|x|\to\infty,$ one has $e^{c_v}\leq \| g  K^{-1}\|_{\ell^\infty(\mathrm{supp} K)}.$\smallskip

{\it Case 2.} $v<0$ on $\Z^2.$ We claim that 
$${\bf F}:=\big\{x\in \Z^2: \Delta v(x)<0\big\}$$
is an infinite set if $F$ is not empty. 

We argue by contradiction and suppose that $F$ is a finite set,  $a:=-\max_{\bf F} v>0.$ Consider the function 
$$v_a:=\max\big\{v, -a\big\}.$$ Since $\Delta v\geq 0$ on $\Z^2\setminus {\bf F}$ and $v_a\equiv -a$ on ${\bf F},$ one can show that $\Delta v_a\geq 0$ on $\Z^2$ and $v_a$ is  bounded  on $\Z^2.$ By Theorem~\ref{thm:para}, $v_a$ is constant. Since $v_a(x)\to 0$ as $|x|\to \infty,$ $v_a\equiv 0,$ which is a contradiction. This proves the claim.
Since $F$ is infinite, we can find a sequence $\{y_i\}_{i=1}^\infty\subset {\bf F}$ such that $y_i\to \infty.$ For each $y_i,$
\begin{align*}
K(y_i) e^{\kappa v(y_i)+c_{v}}< g(y_i). 
\end{align*} By the same argument before, $K(y_i)>0$ and 
$$e^{\kappa v(y_i)+c_{v}}\leq \| g  K^{-1}\|_{\ell^\infty(\mathrm{supp} K)}.$$ Passing to the limit, $i\to \infty,$ since $v<0,$
$$\|e^{\kappa  v+c_{v}} \|_{\ell^\infty}\leq e^{c_{v}}\leq \| g  K^{-1}\|_{\ell^\infty(\mathrm{supp} K)}.$$
Combining the above cases, $\sup_{v\in {\bf A}}\|e^{\kappa  v+c_{v}} \|_{\ell^\infty}<\infty,$ and ${\bf A}$ is a bounded set.

By Theorem \ref{Schaefer fixed-point}, $\cT_1$ has a fixed point in $v_0 \in \ell^\infty_{\tau_1}(\Z^2)$ such that
$$v_0= \Phi_0\ast\big(g-K e^{\kappa  v_0+c_{v_0}}\big) \quad \text{ in $\Z^2$},$$
then we have that 
\begin{align}\label{eq 2.3--ab}
-\Delta v_0= g-K e^{\kappa v_0+c_{v_0}} \quad {\rm in}\ \, \Z^2.  
\end{align}
Set $u=\frac1{\kappa}  c_{v_0}+v_0.$ Then $u$ is the solution of (\ref{eq 2.2-ab}), $u(x)=  \frac1{\kappa} c_{v_0}+O(|x|^{\frac{2-\tau_0}{\tau_0+1}}(\ln |x|)^{\frac{1}{\tau_0+1}} )$ as $x\to+\infty$ by Proposition \ref{pr 2.1} satisfying

$$c_{v_0}\leq \ln \big(\| g  K^{-1}\|_{L^\infty(\mathrm{supp} K)}\big)$$
and 
$$ \int_{\Z^2}  K e^{\kappa u }dx= \int_{\Z^2} g  dx. $$

\smallskip

 \noindent {\it Uniqueness. }  Let $u_1,u_2$ be two solutions of  (\ref{eq 2.2-ab}) satisfying
  $$u_i(x)=a_i+O(|x|^{\frac{2-\tau_0}{\tau_0+1}}(\ln |x|))^{\frac{1}{\tau_0+1}}\big) \quad {\rm as} \ \, |x|\to+\infty, $$
    where $a_i\in\R$ for $i=1,2.$ 
Then $w=u_1-u_2$ satisfies the equation
\begin{equation}\label{eq:diff}-\Delta w+cw=0,\quad \Z^2,\end{equation} where 
$$c(x):=\left\{\begin{array}{ll}
   K(x)\frac{e^{\kappa u_1}-e^{\kappa u_2}}{u_1-u_2}(x),  & u_1(x)\neq u_2(x), \\[1.5mm]
   0,  &  \mathrm{otherwsie}.
\end{array}\right.$$
Since $c\geq 0,$ one can apply the maximum principle for $w$. We claim that $a_1=a_2.$
 Suppose that it is not true, say $a_1>a_2.$ Then $\displaystyle \lim_{|x|\to \infty}w>0.$ By the maximum principle, Theorem~\ref{thm:max}, $w>0$ on $\Z^2.$ Hence $\Delta w=cw\geq 0$ on $\Z^2.$ Since $w$ is bounded, by Theorem~\ref{thm:para}, $w$ is constant, i.e. $w\equiv b$ for some positive constant $b.$ Moreover, by the equation $$0=K(e^{\kappa u_1}-e^{\kappa u_2})=e^{\kappa u_2}(e^b-1).$$ This is a contradiction, which proves the claim.

 Now we prove that $u_1=u_2.$ Since $a_1=a_2,$ $\displaystyle \lim_{|x|\to \infty}w(x)=0.$ By the maximum principle, Theorem~\ref{thm:max}, $w\equiv 0.$ This proves the uniqueness.

 \subsection{ Non-topological solutions }
 
 Recall that $\Phi_0$ be the fundamental solution of $\Delta$ in $\Z^2$ satisfying 
    $\Phi_0(0)=0$,  $\Phi_0<0$ in $\Z^2\setminus\{0\}$ and   
$$
  \Phi_0(x)= -\frac1{2\pi}\ln |x|-\frac{\gamma_0}{2}+O(|x|^{-1})\quad {\rm as}\ \ x\to\infty.
$$

\begin{proposition}\label{pr 3-1-ab}
Assume that  $\kappa>0$ and  $\beta>\frac{4\pi}{\kappa}$. 
 
  Then  for any $\alpha\in  \big(\frac{4\pi}{\kappa}, \beta\big)$ problem (\ref{eq 1.1-ab})
has a solution ${\bf u}_\alpha$ satisfying 
\begin{align}\label{ep-1-ab}
{\bf u}_{\alpha}(x)=-\frac{\alpha}{2\pi}\ln|x|+{\bf d}_{\alpha,\beta}+O(|x|^{\frac{4\pi-\alpha\kappa}{\alpha\kappa+2\pi}}\big(\ln |x|\big)^{\frac{2\pi}{\alpha\kappa+2\pi}}) 
\quad {\rm as}\ \, |x|_{_Q}\in\N\to+\infty,
\end{align}
where ${\bf d}_{\alpha,\beta}\in\R$ depends on $\alpha,\beta$.  The solution is unique under the restriction of the asymptotic behavior 
$$u (x)=-\frac{\alpha}{2\pi}\ln|x|+O(1)\quad {\rm as}\ \, |x|_{_Q}\in\N\to+\infty.$$

Moreover, the mapping $\alpha\mapsto {\bf u}_\alpha$ is strictly decreasing, 
$${\bf d}_{\alpha,\beta}\leq \frac1\kappa \ln(\beta-\alpha)-\frac{\gamma_0}{2}\alpha$$ 
and
\begin{align}\label{ep-1-e-ab}
\int_{\Z^2} e^{\kappa {\bf u}_\alpha} dx=\beta-\alpha. 
\end{align}

   \end{proposition}

\noindent {\bf Proof. }  We first prove the existence of the solution.
We shall find a solution 
${\bf u}_\alpha$ with $\alpha\in  \big(\frac{4\pi}{\kappa}, \beta\big)$  having the asymptotic ${\bf u}_\alpha\sim -\frac{\alpha}{2\pi}\ln(|x|) $
as $x\to\infty.$  Assume that
$${\bf u}_\alpha=\tilde u+  \alpha \Phi_0\quad   \text{ in $\Z^2$}.$$ Then $\tilde u$ is a solution of 
 \begin{align}\label{req-1--ab-0}
  - \Delta \tilde u+  K_{\alpha}  e^{\kappa \tilde u}= g_{\alpha,\beta} \quad
    {\rm in}\ \  \Z^2,
 \end{align}
where 
$$K_{\alpha} (x)=e^{ \alpha\kappa \Phi_{0}(x)},\qquad  g_{\alpha,\beta}=(\beta-\alpha)\delta_0. $$
Note that $K_{\alpha}  (0)=1$, 
 \begin{align*} 
 0<K_{\alpha}  (x)=  e^{\alpha \kappa \Phi_0(x)}\leq  c (1+|x|)^{- \frac{\kappa \alpha}{2\pi}} \quad
    {\rm for}\ \  x\in\Z^2,
 \end{align*}
where  $c>0$.  \medskip
 
   Next we shall obtain a solution by applying Proposition \ref{pr 3.1-ab} with the setting $\tau_0=\frac{\kappa \alpha}{2\pi}$,  $K=K_{\alpha }$ and $g=g_{\alpha,\beta},$ with $K,g\in \ell^\infty_{\tau_0}(\Z^2).$ Note that
$$\| g_{\alpha,\beta}  K_\alpha^{-1}\|_{\ell^\infty(\mathrm{supp} K)} = \frac{g_{\alpha,\beta}(0)}{K_{\alpha}(0)} =\beta-\alpha. $$

 Therefore,  for  any $\alpha\in  \big(\frac{4\pi}{\kappa}, \beta\big)$, 
Proposition \ref{pr 3.1-ab} shows that problem (\ref{req-1--ab-0}) has a unique solution $v_\alpha$ such that 
$$v_\alpha= c_\alpha+O(|x|^{\frac{4\pi- \kappa \alpha }{ \kappa \alpha+2\pi}}\big(\ln |x|\big)^{\frac{2\pi }{ \kappa \alpha +2\pi}}) \quad {\rm as}\ \ x\to\infty,  $$
where
  \begin{align}\label{ess-1--0-ab}
  c_\alpha\leq \frac1\kappa \ln(\beta-\alpha). 
   \end{align}

Let $${\bf u}_\alpha=v_\alpha+ \alpha  \Phi_0, $$
then ${\bf u}_\alpha$ is a solution  problem (\ref{eq 1.1-ab})  having 
$$
{\bf u}_\alpha(x)=-\frac{\alpha}{2\pi}\ln|x|+{\bf d}_{\alpha,\beta}+O(|x|^{\frac{4\pi-\alpha\kappa}{\alpha\kappa+2\pi}}\big(\ln |x|\big)^{\frac{2\pi}{\alpha\kappa+2\pi}}) 
\quad {\rm as}\ \, |x|_{_Q}\in\N\to+\infty,
$$
where 
 \begin{align*} 
 {\bf d}_{\alpha,\beta}=c_\alpha -\alpha \frac{\gamma_0}{2}\leq \frac1\kappa \ln(\beta-\alpha)-\frac{\gamma_0}{2}\alpha.   
 \end{align*}
 Moreover, we see that 
 \begin{align*}
\int_{\Z^2} e^{\kappa  {\bf u}_\alpha} dx=\int_{\Z^2} K_{\alpha}  e^{\kappa \tilde u}dx = \int_{\Z^2} g_{\alpha,\beta} dx=\beta-\alpha. 
\end{align*}
The uniqueness follows by Proposition \ref{pr 3.1-ab} and the decreasing monotonicity $\alpha\mapsto {\bf u}_\alpha$ follows from 
the maximum principle, Theorem~\ref{thm:max}, and \eqref{eq:diff}.    \hfill$\Box$

\subsection{Extremal solutions}

\begin{proposition}\label{pr 3-2-ab}
  Let $\kappa>0$ and  $\beta>\frac{4\pi}{\kappa}$, 
 then  problem (\ref{eq 1.1-ab})
has a solution ${\bf u}_0$ satisfying (\ref{ep-1-0-ab}). 
 \end{proposition}

In the critical case, we need to involve special functions to construct super and sub solutions. To this end,
we let $Q_n=\big\{ x\in\Z^2:\, |x|_{_Q}\leq n \big\}$ and 
 $$\Lambda_{ 0}(x)=
  \begin{cases}
\ln\ln(\frac12+|x|^2)\ \,\, &\text{ for } \  |x|\geq e^2\\[1mm]
 0\ \,\, &\text{ for } \  |x|< e^2.
\end{cases}
  $$ 
  Then 
   \begin{align*} 
 \int_{Q_n} (-\Delta) \Lambda_{ 0}(x)dx= -\int_{\delta Q_n} \frac{\partial\Lambda_{ 0}}{\partial n} (x)dx, 
 \end{align*}
 where $$\frac{\partial\Lambda_{ 0}}{\partial n} (x)=\sum_{y\in Q_n:y\sim x}\left(\Lambda_{ 0}(x)-\Lambda_{ 0}(y)\right).$$ Hence
    \begin{align*} 
 \big| \int_{\delta Q_n} \frac{\partial\Lambda_{ 0}}{\partial n} (x)dx\big|  dx
& \leq c\frac{n}{(\frac12+n^2)\ln(\frac12+n^2)}   \big|\delta Q_n  \big| 
 \\&\leq c\frac{4n^2}{(\frac12+n^2)\ln(\frac12+n^2)} \to 0\quad{\rm as}\ \, n\to+\infty. 
  \end{align*}

 As a consequence, 
  \begin{align}  
 \int_{\Z^2} (-\Delta) \Lambda_{ 0}(x)dx =\lim_{n\to+\infty} \int_{Q_n} (-\Delta) \Lambda_{ 0}(x)dx=0. 
 \end{align}
Let $\varphi_0(t)=\ln\ln (\frac12+t)$, then 
$\varphi_0'(t)=\frac1{(\frac12+t)\ln (\frac12+t)}$, 
$$\varphi_0''(t)=-\frac1{(\frac12+t)^2\ln (\frac12+t)}-\frac1{(\frac12+t)^2\big(\ln (\frac12+t)\big)^2},$$
$$\varphi_0'''(t)=\frac1{(\frac12+t)^3\ln (\frac12+t)}+\frac3{(\frac12+t)^3\big(\ln (\frac12+t)\big)^2}+\frac2{(\frac12+t)^3\big(\ln (\frac12+t)\big)^3}, $$
and
 \begin{align*} 
 \varphi_0^{(4)}(t)=&-\Big(\frac6{(\frac12+t)^4\ln  (\frac12+t) }+\frac{11}{(\frac12+t)^4\big(\ln  (\frac12+t)\big)^2}+\frac9{(\frac12+t)^4\big(\ln (\frac12+t)\big)^3}
 \\[1mm]&\qquad +\frac6{(\frac12+t)^4\big(\ln  (\frac12+t) \big)^4}\Big),
  \end{align*} 
then
 \begin{align*} 
 \Delta  \Lambda_{ 0}(x) &=\sum_{y\sim x} \big(\Lambda_{ 0}(y)-\Lambda_{ 0}(x)\big)
 \\&=\sum_{y\sim x}\Big(\varphi_0'(|x|^2) (|y|^2-|x|^2)+\frac12 \varphi_0''(|x|^2)(|y|^2-|x|^2)^2+\frac16 \varphi_0'''(|x|^2)(|y|^2-|x|^2)^3 
 \\&\qquad  + \frac1{24} \varphi_0^{(4)}(|  x_y|^2)(|y |^2-|x |^2)^4 \Big) 
 \\&= \frac{4}{(\frac12+|x|^2)\ln(e+|x|^2)} -\frac{4|x|^2+2}{(\frac12+|x|^2)^2\ln(e+|x|^2)}-\frac{4|x|^2+2}{ (\frac12+|x|^2)^2\big(\ln(\frac12+|x|^2)\big)^2} 
 \\&\qquad +\frac13\frac{2+12|x|^2}{(\frac12+|x|^2)^3\ln(e+|x|^2)}
 +  \frac{2+12|x|^2}{ (\frac12+|x|^2)^3\big(\ln(\frac12+|x|^2)\big)^2} 
 \\&\qquad +\frac23 \frac{2+12|x|^2}
 { (\frac12+|x|^2)^3\big(\ln(\frac12+|x|^2)\big)^3}  +W_0(x)
 \\&= -\frac{1}{ (\frac12+|x|^2) \big(\ln(\frac12+|x|^2)\big)^2}  \frac{6|x|^2+1}{2|x|^2+1}
 \bigg(\frac{2|x|^2+1}{6|x|^2+1}-\frac13 \frac{ \ln(\frac12+|x|^2)}{  \frac12+|x|^2}- \frac{1}{  \frac12+|x|^2}
  \\&\qquad\qquad -\frac23   \frac{1}{  (\frac12+|x|^2)\ln(\frac12+|x|^2) }\bigg)
   +W_0(x)
    \end{align*}
where for $|x|\geq 10$, we have that 
$$\frac13 -\frac1{10}<\frac{2|x|^2+1}{6|x|^2+1}-\frac13 \frac{ \ln(\frac12+|x|^2)}{  \frac12+|x|^2}- \frac{1}{  \frac12+|x|^2}
 -\frac23   \frac{1}{  (\frac12+|x|^2)\ln(\frac12+|x|^2) } <\frac13$$
and 
\begin{align*}  
|W_0( x)| &= \Big|\sum_{y\sim x} \big(\frac1{24} \varphi_0^{(4)}(|  x_y|^2)(|y |^2-|x |^2)^4\big)\Big|
\\&\leq   \frac1{24}  \frac{2^5 ( x_1^4+x_2^4)+24(x_1^2+x_2^2)+4 }{(\frac12+|x_y|^2)^4\ln   (\frac12+|x_y|^2) }
\\& \leq \frac1{6} \frac{ 8 |x|^4 +6 |x|^2 +1}{(|x|^2-\frac12)^4 \ln ( |x|^2-\frac12) }
\\& = \frac1{3} \Big(\frac{  4}{(|x|^2-\frac12)^2 \ln ( |x|^2-\frac12) }+ \frac{ 7}{(|x|^2-\frac12)^3 \ln ( |x|^2-\frac12) }+\frac{ 3}{(|x|^2-\frac12)^4 \ln ( |x|^2-\frac12) }\Big)
\\&<   \frac{  5}{(|x|^2-\frac12)^2 \ln ( |x|^2-\frac12) } 
\\&<   \frac{  10 }{(|x|^2+\frac12)^2 \ln ( |x|^2+\frac12) }, 
   \end{align*}
 since
   $$ x_y\in \big\{z\in \R^2:\,  z=x+t(y-x)\quad{\rm for}\  t\in[0,1]\big\}.$$ Here we have that $ |x|-1\leq |x_y|\leq |x|+1.$
  
   As a consequence, there exists $m_0\geq 10$ such that   for $|x|\geq m_0$
   \begin{equation}\label{test-1}
   - \frac{2}{ (\frac12+|x|^2)  \big(\ln (\frac12+|x|^2)\big)\big)^2}\leq \Delta  \Lambda_{0}(x) \leq -\frac12 \frac{1}{ (\frac12+|x|^2)  \big(\ln ( \frac12+|x|^2 )\big)^2} 
   \end{equation}
    and   for $|x|\leq m_0$ there exists $d_0>0$ such that 
    \begin{equation}\label{test-2}
    | \Delta  \Lambda_{0}(x)|\leq d_0.
       \end{equation}

\medskip

\noindent{\bf Proof of Proposition  \ref{pr 3-2-ab}. }  {\it Existence of the extremal solution. }  In the extremal case $\alpha_0 =\frac{4\pi}\kappa$ , we shall find a solution 
${\bf u}_{ 0}$   having the asymptotic 
$${\bf u}_{ 0}\sim -\frac{2}\kappa\ln |x| -\frac{2}{\kappa} \Lambda_{0}(x)+O(1) \quad  \text{as $|x|_{_Q}\in\N\to+\infty$}. $$

We shall obtain solutions of (\ref{eq 1.1-ab}) via Perron's method by construct suitable  sub solutions. 
To this end, we denote 
$$u_d= \alpha_0 \Phi_0-\frac{2}{\kappa} \Lambda_{0}+\frac{d}{\kappa}  \quad   \text{ in  $\Z^2$,}$$
then 
 \begin{align}\label{req-1--ab}
  - \Delta  u_d+ e^{\kappa u_d}-\beta \delta_0=  K_{d} - g_0 \quad
    {\rm in}\ \  \Z^2,
 \end{align}
where  $\alpha_0=\frac{4\pi}{\kappa}$, 
$$K_{ d} (x)=e^{d} e^{ \alpha_0\kappa \Phi_0(x) -2  \Lambda_0(x) } ,\qquad  g_0=(\beta-\alpha_0 )\delta_0+\frac{2}{\kappa} (-\Delta) \Lambda_{0} \quad{\rm in}\ \Z^2. $$
 
 Note that 
  $$ g_0= \frac{2}{\kappa}(- \Delta) \Lambda_{0}  \quad {\rm in}\ \, \Z^2\setminus\{0\},$$ 
 then (\ref{test-1}) implies that   
 \begin{align*}  
\frac1c \frac{1}{ (\frac12+|x|^2) \big(\ln(\frac12+|x|^2)\big)^2}  \leq  g_0(x)\leq c \frac{1}{ (\frac12+|x|^2) \big(\ln(\frac12+|x|^2)\big)^2}\quad \text{for  $|x|\geq m_0$,}
 \end{align*}
 (\ref{test-2}) leads to  
  $$|g_0(x)|\leq c\qquad \text{for  \ $|x|\leq m_0$}$$
and
$$
\int_{\Z^2}g_0(x)dx=\beta  -\frac{4\pi}{\kappa} >0.
$$

 Note that   
$$
 K_{d} (x)=e^d   e^{  \alpha_0 \kappa \Phi_0(x)-2 \Lambda_0(x)},
 $$
   then for some $c\geq 1$ independent of $d$ such that 
 $$
\frac1c e^d    (1+|x|)^{-2} \big(\ln(e+|x|)\big)^{-2} \leq  K_{ d}  (x)\leq c e^d    (1+|x|)^{-2} \big(\ln(e+|x|)\big)^{-2}.  
$$

Then there exists $d_1>0 $ such that 
$$ - \Delta  u_{d_1}+ e^{ \kappa u_{d_1}} -\beta \delta_0\geq 0. 
$$
and there exists  $d_2<0$ such that 
$$- \Delta  u_{d_2}+  e^{ \kappa u_{d_2}} -\beta \delta_0\leq 0\quad {\rm for}\ |x|\geq m_0$$
and 
$$u_{d_2}(x)\leq u_{\frac{4\pi+\beta\kappa}{2\kappa}}(x)\quad {\rm for}\ \ |x|\leq m_0,  $$
where $u_{\frac{4\pi+\beta\kappa}{2\kappa}}$ is the solution derived in Proposition \ref{pr 3-1-ab} with $\frac{4\pi+\beta\kappa}{2\kappa}\in (\frac{4\pi}{\kappa},\beta)$.
Obviously, $u_{d_1}> u_{d_2}$ in $\Z^2$ and there exists an integer $n_0>m_0$ such that 
$$ u_{d_2}(x)> u_{\frac{4\pi+\beta\kappa}{2\kappa}}(x)\quad {\rm for}\ \, |x|\geq n_0. $$ 

 Now we let 
 $$w_0=\max\big\{u_{\frac{4\pi+\beta\kappa}{2\kappa}} ,u_{d_2}\big\}\quad {\rm in}\ \, \Z^2.$$
 We claim that $w_0$ is a sub-solution of (\ref{eq 1.1-ab}). 
 In fact, let 
 $${\bf E}_+=\big\{x\in\Z^2:\,  u_{d_2}(x)\leq u_{\frac{4\pi+\beta\kappa}{2\kappa}}(x) \big\},  $$
 then 
 $$B_{m_0}(0)\subset {\bf E}_+\subset B_{n_0}(0). $$
  For $x\in \Z^2\setminus {\bf E}_+$, $w_0(x) =u_{d_2}(x)$, 
   \begin{align*}
 -\Delta w_0(x)  &= \sum_{y\sim x} \big( u_{d_2}(x)-w_0(y)\big) 
 \\[1mm]& \leq  \sum_{y\sim x} \big( u_{d_2}(x)-u_{d_2}(y)\big)
=-\Delta u_{d_2}(x) 
 \\[1mm]& \leq -e^{\kappa u_{d_2}(x)}+\beta\delta_0
   \\[1mm]&  =-e^{\kappa w_0(x)}+\beta\delta_0
 \end{align*}
 and for $x\in {\bf E}_+$, $w_0(x) =u_{\frac{4\pi+\beta\kappa}{2\kappa}}(x)$, 
   \begin{align*}
 -\Delta w_0(x)  &= \sum_{y\sim x} \big( u_{\frac{4\pi+\beta\kappa}{2\kappa}}(x)-w_0(y)\big) 
 \\[1mm]& \leq  \sum_{y\sim x} \big( u_{\frac{4\pi+\beta\kappa}{2\kappa}} (x)-u_{\frac{4\pi+\beta\kappa}{2\kappa}} (y)\big)
=-\Delta u_{\frac{4\pi+\beta\kappa}{2\kappa}}(x)  
 \\[1mm]& \leq -e^{\kappa u_{\frac{4\pi+\beta\kappa}{2\kappa}}(x)}+\beta\delta_0
  =-e^{\kappa w_0(x)}+\beta\delta_0.
 \end{align*}
 Therefore, $w_0$ is a sub solution of (\ref{eq 1.1-ab}). By selecting a sufficiently large $d_1$ such that  
$u_{\frac{4\pi+\beta\kappa}{2\kappa}} \leq u_{d_1}$ in $B_{n_0}(0)$, we further obtain that $w_0 \leq u_{d_1}$ in $\Z^2$.

Let 
  $w_n\ (n\geq 1)$ be the solution of  non-homogeneous problem {
  \begin{equation}\label{eq 2.1-ab}
\left\{
\begin{array}{lll}
- \Delta  u + L u =Lw_{n-1} -e^{\kappa w_{n-1}} + \beta \delta_0,  \quad
   &{\rm in}\ \ B_{n+n_0}(0), \\[2mm]
 \phantom{ -- }
u=w_0,  \quad  &{\rm   in} \ \  \Z^2\setminus B_{n+n_0}(0), 
 \end{array}
 \right.
\end{equation}
where   $n\in\N$ and $L>0$ is such that  the function $t\mapsto Lt -e^{\kappa t}$ is increasing for $t\leq \max_{x\in\Z^d}u_{d_1}(x)$. }

 
 Since $w_1=  w_0$ in $\Z^d\setminus B_{1+n_0}(0)$ and $w_0, u_{d_1}$ are  sub-solution and super-solution  of (\ref{eq 2.1-ab}) with $n=1$ respectively, then we apply  Theorem \ref{thm:max}
 with $c=L$  to  obtain that $w_0\leq  w_1\leq u_{d_1}$ in $B_{n_0+1}(0)$ and
 $$u_{d_2}\leq w_0\leq  w_1\leq u_{d_1}\quad {\rm in}\ \,  \Z^2. $$ 
Iteratively,  we obtain that 
 $$u_{d_2}\leq w_0\leq  w_1\leq \cdots\leq w_n\leq \cdots \leq  u_{d_1}\quad  {\rm   in} \ \  \Z^2.$$
 Denote 
 $${\bf u}_0:= \lim_{n\to+\infty} w_n\quad  {\rm   in} \ \  \Z^2. $$
  Then ${\bf u}_0$ is a solution of (\ref{eq 1.1-ab}) satisfying 
  $$u_{d_2}\leq {\bf u}_0\leq u_{d_1} \quad  {\rm   in} \ \  \Z^2 $$
i.e.
  $${\bf u}_0= \alpha_0 \Phi_0-\frac{2}{\kappa} \Lambda_0+O(1).$$
The proof ends.  \hfill$\Box$\smallskip

  \subsection{Properties of ${\bf u}_\alpha$}
  \begin{lemma}\label{lm 3.1-subsec-1-ab}
  Let ${\bf u}_0$ be the extremal solution of (\ref{eq 1.1-ab}) derived from Proposition  \ref{pr 3-2-ab} and
  $${\bf u}_{\alpha_0}=\lim_{\alpha\to\alpha_0^+} {\bf u}_\alpha\quad{\rm in}\ \, \Z^2,$$
  then
${\bf u}_0={\bf u}_{\alpha_0}$ in $\Z^2$  
  and
  $$\int_{\Z^2} e^{\kappa {\bf u}_0}dx=\beta-\frac{4\pi}{\kappa}. $$
  \end{lemma}
\noindent  {\bf Proof. } Since for $\alpha>\alpha_0$, 
  $$\lim_{|x|\to+\infty} \frac{{\bf u}_\alpha(x)}{\ln |x|}=-\frac{\alpha}{2\pi},$$
  then comparison principle implies that
$$
 {\bf u}_\alpha\leq {\bf u}_0\quad{\rm in}\ \, \Z^2, 
$$
  where ${\bf u}_0$ is the extremal solution of (\ref{eq 1.1-ab}) derived from Proposition  \ref{pr 3-2-ab} with the behavior (\ref{ep-1-0-ab}) at infinity.  Then we have  that    
 \begin{align}\label{boud-dd} 
 {\bf u}_{\alpha_0}\leq  {\bf u}_0 \quad{\rm in}\ \, \Z^2. 
   \end{align}
 
{
For any $\epsilon>0$, let 
$$w_{\epsilon}={\bf u}_{\alpha_0}-\epsilon(\Phi_0-t_0) \quad{\rm in}\ \, \Z^2, $$
where 
$$t_0=\frac1{\kappa}  e^{-\kappa {\bf u}_{\alpha_0}(0) }>0.$$

Let 
$$h_{\epsilon}(x)= e^{\kappa ({\bf u}_{\alpha_0}-\epsilon(\Phi_0-t_0))}- e^{\kappa {\bf u}_{\alpha_0} }-\epsilon\delta_0,\quad\forall\, x\in\Z^2,$$
 then 
 $$h_{\epsilon}(x) =e^{\kappa ({\bf u}_{\alpha_0}(x)-\epsilon(\Phi_0(x)-t_0))}- e^{\kappa {\bf u}_{\alpha_0}(x) }\geq 0 \quad {\rm for}\ \, x\in\Z^2\setminus\{0\}$$
 and
 \begin{align*} 
 h_{\epsilon}(0)    =e^{\kappa ({\bf u}_{\alpha_0}(0)+\epsilon t_0)}- e^{\kappa {\bf u}_{\alpha_0}(0) }-\epsilon
  &= e^{\kappa {\bf u}_{\alpha_0}(0) }\big(e^{\kappa t_0\epsilon  }-1-\epsilon e^{-\kappa {\bf u}_{\alpha_0}(0) } \big)
 \\[1mm] & \geq e^{\kappa {\bf u}_{\alpha_0}(0) } \epsilon \big(\kappa t_0-  e^{-\kappa {\bf u}_{\alpha_0}(0) } \big)
  \\[1mm] &=0,
 \end{align*}
 since $\Phi_0(0)= 0$, and $\Phi_0(x)<0$ in $\Z^2\setminus\{0\}$. 
Now we can get that  
 \begin{align*} 
  - \Delta  w_\epsilon+ e^{\kappa w_\epsilon}& \geq - \Delta {\bf u}_{\alpha_0 }+ e^{\kappa ({\bf u}_{\alpha_0}-\epsilon(\Phi_0-t_0))} -\epsilon\delta_0
\\[1mm] &=  h_{\epsilon}+\beta \delta_0 
  \\[1mm]&\geq \beta \delta_0 
 =- \Delta {\bf u}_0+ e^{\kappa {\bf u}_0} \quad{\rm in}\ \ \Z^2.
 \end{align*}
 Moreover, by the monotonicity, for any $\alpha\in(\alpha_0,\alpha_0+\epsilon),$ 
 $$\lim_{|x|\to\infty}(w_{\epsilon}-{\bf u}_0)(x)\geq \lim_{|x|\to\infty} ({\bf u}_{\alpha}-\epsilon(\Phi_0-t_0)-{\bf u}_0)(x)=+\infty.$$ 
Hence, by Corollary \ref{cr com} with $f(x,t)=e^{\kappa t}$,  implies that 
  $$w_\epsilon \geq {\bf u}_0\quad{\rm in}\ \, \Z^2. $$
 By the arbitrary of $\epsilon$, 
 then ${\bf u}_{\alpha_0}\geq {\bf u}_0$ in $\Z^2$,
 which, together with (\ref{boud-dd}), implies that 
 $$ {\bf u}_{\alpha_0}= {\bf u}_0\quad{\rm in}\ \, \Z^2. $$

 Note that  
\begin{align*}
\int_{\mathbb{Z}^2} e^{\kappa \mathbf{u}_0}\,dx 
&\leq C\lim_{n\to+\infty} \int_{Q_n} (1+|x|)^{-2} \big(\ln(e+|x|)\big)^{-2}\,dx < +\infty,
\end{align*}
  the sequence $\{\mathbf{u}_\alpha\}_\alpha$ is bounded from above by $\mathbf{u}_0$ and   $\mathbf{u}_\alpha \to \mathbf{u}_{\alpha_0}$ as $\alpha \to \alpha_0^+$  pointwisely. 
Since $e^{\kappa \mathbf{u}_\alpha} \leq e^{\kappa \mathbf{u}_0} \in \ell^1(\mathbb{Z}^2)$), and the identity $\int_{\mathbb{Z}^2} e^{\kappa \mathbf{u}_\alpha}\,dx = \beta - \alpha$ for $\alpha < \alpha_0$, then we obtain that  
\begin{align*}
\int_{\mathbb{Z}^2} e^{\kappa \mathbf{u}_{0}}\,dx 
= \lim_{\alpha \to \alpha_0^-} \int_{\mathbb{Z}^2} e^{\kappa \mathbf{u}_\alpha}\,dx 
= \lim_{\alpha \to \alpha_0^-} (\beta - \alpha) 
= \beta - \alpha_0.
\end{align*}
We complete the  proof.  \hfill$\Box$}\medskip

\begin{corollary}\label{cr 3.1-subsec-1-ab}
$\kappa>0$,  $\beta>\frac{4\pi}{\kappa}$ and $\{{\bf u}_\alpha\}_{\alpha\in[\alpha_0,\beta)}$ be the solutions of (\ref{eq 1.1-ab})  derived in Proposition \ref{pr 3-1-ab}. 
  Then for any $\tilde \alpha\in (\alpha_0,\beta)$, we have 
 $${\bf u}_{\tilde \alpha}=\lim_{\alpha\to \tilde \alpha} {\bf u}_\alpha\quad{\rm locally\ in}\ \, \Z^2.$$
 \end{corollary}
\noindent{\bf Proof. } { By the decreasing monotonicity of the mapping $\alpha\in (\alpha_0,\beta)\mapsto {\bf u}_\alpha$, we can set 
 $${\bf u}_{\tilde \alpha,+}=\lim_{\alpha\to\tilde \alpha ^+} {\bf u}_\alpha\quad{\rm in}\ \, \Z^2, $$
 which is a solution of  (\ref{eq 1.1-ab}) with $\alpha=\tilde \alpha$ and 
  \begin{align} 
{\bf u}_{\tilde \alpha,+} \leq {\bf u}_{\tilde \alpha} \quad{\rm in}\ \, \Z^2.  
  \end{align}
  We need to prove that $${\bf u}_{\tilde \alpha,+} \equiv {\bf u}_{\tilde \alpha} \quad \mathrm{in\ } \Z^2.$$ Suppose that it is not true, then 
\begin{align}\label{boud-dd1}
 \int_{\Z^2} e^{\kappa  {\bf u}_{\tilde \alpha,+}} dx  <   \int_{\Z^2} e^{\kappa {\bf u}_{\tilde \alpha}}dx 
 = \beta- \tilde \alpha. 
 \end{align}
Again by the decreasing monotonicity of the mapping $\alpha\in (\alpha_0,\beta)\mapsto {\bf u}_\alpha$,  there holds  that 
$${\bf u}_{\tilde \alpha,+}\geq {\bf u}_\alpha\quad {\rm for}\ \alpha\in (  \tilde \alpha,\beta), $$
then
\begin{align*}
 \int_{\Z^2} e^{\kappa {\bf u}_{\tilde \alpha,+}}dx  \geq \lim_{\alpha \to \tilde\alpha ^+} \int_{\Z^2} e^{\kappa {\bf u}_{\alpha}}dx 
 = \lim_{\alpha \to \tilde \alpha^+}  (\beta -\alpha)
 = \beta- \tilde \alpha 
 \end{align*}
which contradicts (\ref{boud-dd1}). } Hence $$ \lim_{\alpha\to \tilde \alpha^+} {\bf u}_\alpha={\bf u}_{\tilde \alpha}\quad{\rm in}\ \, \Z^2. $$
  
Similarly,  we can show that 
  $$ \lim_{\alpha\to \tilde \alpha^-} {\bf u}_\alpha={\bf u}_{\tilde \alpha}\quad{\rm in}\ \, \Z^2. $$
The proof ends.\hfill$\Box$ \medskip

\noindent {\bf Proof of Theorem \ref{teo 1-ab}. }  Problem (\ref{eq 1.1-ab})  has solutions  $\{{\bf u}_\alpha\}_{\alpha\in[\alpha_0,\beta)}$  in Proposition \ref{pr 3-1-ab},  has an extremal solution ${\bf u}_{\alpha_0}$ in Proposition \ref{pr 3-2-ab} and the properties of solutions in $(iii)$ are proved in Lemma \ref{lm 3.1-subsec-1-ab} and Corollary \ref{cr 3.1-subsec-1-ab}. 
 \hfill$\Box$

  \bigskip

   \bigskip

\bigskip

{\footnotesize 
\noindent {\bf  Conflicts of interest:} The authors declare that they have no conflicts of interest regarding this work.

\smallskip 

\noindent {\bf  Data availability:}  This paper has no associated data.\smallskip

\noindent {\bf Acknowledgements:} We thank the referees for their detailed reports and valuable suggestions, which greatly improved our writing. We would like to thank Genggeng Huang, Yong Lin, Yunyan Yang for helpful discussions on Kazdan-Warner equations on graphs.
 H. Chen is supported by  NSFC,  No.  12361043.
 B. Hua is supported by NSFC, No.12371056.

}

\bibliographystyle{alpha}
\bibliography{KW2025}
            
  \end{document}